\documentclass[10pt]{amsart}
\usepackage{amsthm}
\usepackage{amstext}
\usepackage{enumitem}
\usepackage{amssymb}
\usepackage{amsmath,calligra,mathrsfs}
\usepackage[all,cmtip]{xy}
\usepackage{dsfont}
\usepackage{hyperref}
\hypersetup{
	colorlinks   = true,          
	urlcolor     = blue,          
	linkcolor    = blue,          
	citecolor   = blue             
}
\usepackage{graphicx}
\newcommand\cal{\mathcal}
\newcommand\bb{\mathbb}

\usepackage{comment}
\usepackage{xcolor}
\usepackage{tikz-cd}
\usepackage{mathtools}

\theoremstyle{plain}
\newtheorem{thm}{Theorem}[section]
\newtheorem{lem}[thm]{Lemma}
\newtheorem{prop}[thm]{Proposition}
\newtheorem{cor}[thm]{Corollary}

\theoremstyle{definition}
\newtheorem{defn}[thm]{Definition}
\newtheorem{exmp}[thm]{Example}

\theoremstyle{remark}
\newtheorem{rem}[thm]{Remark}
\newtheorem{rems}[thm]{Remarks}
\theoremstyle{plain}

\DeclareMathOperator{\id}{id}

\DeclareMathOperator{\Proj}{Proj}

\DeclareMathOperator{\Gal}{Gal}

\DeclareMathOperator{\sheafhom}{\mathscr{H}\text{\kern -3pt {\calligra\large om}}\,}

\DeclareMathOperator{\Pic}{Pic}

\DeclareMathOperator{\Br}{Br}
\DeclareMathOperator{\NS}{NS}
\DeclareMathOperator{\CH}{CH}

\DeclareMathOperator{\bCH}{\mathbf{CH}}
\DeclareMathOperator{\bPic}{\mathbf{Pic}}

\DeclareMathOperator{\Hilb}{Hilb}
\DeclareMathOperator{\Sym}{Sym}

\DeclareMathOperator{\fppf}{fppf}

\DeclareMathOperator{\bl}{bl}

\DeclareMathOperator{\bAlb}{\mathbf{Alb}}
\DeclareMathOperator{\alg}{alg}

\newcommand{\defi}[1]{\textsf{#1}} 

\newcommand{\kbar}{{\overline{k}}}

\def\Ker{\text{Ker}}

\let\phi\varphi
\def\P{\mathbb{P}}
\def\O{\mathcal{O}}
\def\Z{\mathbb{Z}}
\def\C{\mathbb{C}}

\def\R{\mathbb{R}}
\def\F{\mathbb{F}}

\title[\tiny 
Arithmetic and birational properties of linear spaces on intersections of two quadrics
]{Arithmetic and birational properties of linear spaces on intersections of two quadrics}
\author{Lena Ji}
\address{Department of Mathematics, University of Illinois Urbana-Champaign, 273 Altgeld Hall, 1409 W. Green Street, Urbana, IL 61801}
\email{lenaji.math@gmail.com}
\urladdr{https://lji.web.illinois.edu/}
\author{Fumiaki Suzuki}
\address{Institute of Algebraic Geometry, Leibniz University Hannover, Welfengarten 1, 30167, Hannover, Germany}
\email{suzuki@math.uni-hannover.de}
\urladdr{https://fumiaki-suzuki.github.io/}


\thanks{
During the preparation of this manuscript,
L.J. was supported in part by NSF MSPRF grant DMS-2202444.
F.S. is supported by the ERC grant “RationAlgic” (grant no. 948066) and partially by the DFG project EXC-2047/1 (project no. 390685813).
}

\begin{document}

\begin{abstract}
We study rationality questions for Fano schemes of linear spaces on smooth complete intersections of two quadrics, especially over non-closed fields. Our approach is to study hyperbolic reductions of the pencil of quadrics associated to $X$. We prove that the Fano schemes $F_r(X)$ of $r$-planes are birational to symmetric powers of hyperbolic reductions, generalizing results of Reid and Colliot-Th\'el\`ene--Sansuc--Swinnerton-Dyer, and we give several applications to rationality properties of $F_r(X)$.

For instance, we show that if $X$ contains an $(r+1)$-plane over a field $k$, then $F_r(X)$ is rational over $k$. When $X$ has odd dimension, we show a partial converse for rationality of the Fano schemes of second maximal linear spaces, generalizing results of Hassett--Tschinkel and Benoist--Wittenberg. When $X$ has even dimension, the analogous result does not hold, and we further investigate this situation over the real numbers. In particular, we prove a rationality criterion for the Fano schemes of second maximal linear spaces on these even-dimensional complete intersections over $\mathbb R$; this may be viewed as extending work of Hassett--Koll\'ar--Tschinkel.
\end{abstract}

\maketitle

\section{Introduction}

Over an arbitrary field $k$ of characteristic $\neq 2$,
let $X$ be a smooth complete intersection of two quadrics in $\P^N$.
There is an extensive literature on 
the Fano schemes \(F_r(X)\) of
\(r\)-dimensional linear spaces on such complete intersections 
\cite{Gauthier54,Reid-thesis,Tyurin-survey, Donagi80, Wang-maximal-linear-spaces},
and there have been many applications to arithmetic problems \cite{CTSSD, BGW, IP-hyperelliptic, CT-intersection-quadrics}, moduli theory \cite{Desale-Ramanan, Ramanan81}, and rationality questions \cite{ABB14,HT-intersection-quadrics,BW-IJ}.
In this paper, we study birational properties of
the Fano schemes of linear spaces on $X$ via 
hyperbolic reductions of the pencil of quadrics $\mathcal{Q}\rightarrow \P^1$ associated to \(X\), with an eye toward applications to rationality questions over non-closed fields.

A variety over a field \(k\) is \defi{\(k\)-rational} if it is birationally equivalent to projective space over \(k\).
It is classically known that if a smooth complete intersection of two quadrics \(X \subset \bb P^N\) contains a line defined over $k$, then $X$ is $k$-rational by projection from this line (see, e.g., \cite[Proposition 2.2]{CTSSD}).
Our first result is a generalization of this to Fano schemes of higher-dimensional linear spaces on \(X\):
\begin{thm}\label{Fanoscheme}
Over a field $k$ of characteristic $\neq 2$, 
let $X$ be a smooth complete intersection of two quadrics in $\P^N$.
Let $0\leq r\leq \lfloor \frac{N}{2}\rfloor -2$.
If $F_{r+1}(X)(k)\neq \emptyset$, 
then $F_r(X)$ is $k$-rational.
\end{thm}

The Fano scheme \(F_r(X)\) is non-empty exactly for \(0 \leq r \leq \lfloor \frac{N}{2}\rfloor - 1\), and the Fano scheme of maximal linear subspaces \(F_{\lfloor \frac{N}{2}\rfloor - 1}(X)\) is always irrational. So the range in Theorem~\ref{Fanoscheme} is optimal.

An immediate corollary of Theorem~\ref{Fanoscheme} over \emph{algebraically closed} fields is that the Fano schemes of non-maximal linear spaces on $X$ are all rational (Corollary \ref{cor:fano-scheme-rational}\eqref{item:Ci}); that is, \(F_r(X)\) is rational for all \(0\leq r\leq \lfloor \frac{N}{2}\rfloor -2\).
To the authors' knowledge, this result was previously only known for \(r=0\) and \(r = \lfloor\frac{N}{2}\rfloor-2\) (the latter case by combining works of \cite{Desale-Ramanan,Newstead75,Newstead80,bauer91,casagrande15}).
For $0<r<\lfloor\frac{N}{2}\rfloor-2$, only unirationality of $F_r(X)$ for general $X$ was previously known by Debarre--Manivel \cite{DebarreManivel98}.

One may wonder whether the converse of Theorem~\ref{Fanoscheme} holds.
In general, the answer to this question is no: counterexamples for $(r,N)=(0, 4), (0, 6)$ are known over \(\bb R\) by \cite[Remark 37]{HT-intersection-quadrics} and \cite[Propositions 6.1 and 6.2]{HassettKoll'arTschinkel}.
We additionally show that there are counterexamples for \((r,N)=(g-2, 2g)\) for any \(g\geq 2\) (see Corollary~\ref{cor:fano-R-rationality-connected-real-locus-isotopy}\eqref{item:even-maximallinearspace-fails}).

However, for $(r,N)=(0,5)$, the above question has a positive answer over any field. In this case, Hassett--Tschinkel (over $\R$) \cite{HT-intersection-quadrics} and Benoist--Wittenberg (over arbitrary $k$) \cite{BW-IJ} show that \(F_0(X) = X \subset \bb P^5\) is \(k\)-rational if and only if \(F_1(X)(k)\neq\emptyset\).
In this case, lines are the maximal linear subspaces on the threefold \(X\). More generally, when \(N=2g+1\) is odd, the Fano schemes of linear subspaces on \(X\) encode a lot of interesting arithmetic and geometric data. Over \(\kbar\), Weil first observed that
the Fano scheme \(F_{g-1}(X)\) of maximal linear subspaces is isomorphic to the Jacobian of the genus \(g\) hyperelliptic curve obtained as the Stein factorizaton of $F_g(\mathcal{Q}/\P^1)\rightarrow\P^1$ \cite{Gauthier54} (see also \cite{Reid-thesis,Desale-Ramanan,Donagi80}). 
Over \(k\), Wang studied the torsor structure of \(F_{g-1}(X)\) \cite{Wang-maximal-linear-spaces}. For the second maximal linear subspaces, \(F_{g-2}(X)\) is isomorphic over \(\kbar\) to the moduli space of rank 2 vector bundles on the hyperelliptic curve \cite{Desale-Ramanan}.

The main result of this paper
proves a partial converse to Theorem~\ref{Fanoscheme} in the case when \(N=2g+1\) and \(r=g-2\), generalizing Hassett--Tschinkel and Benoist--Wittenberg's results to arbitrary \(g \geq 2\).
To state this result, we first need to introduce a definition. If $F_{r}(X)(k)\neq\emptyset$, choose $\ell \in F_r(X)(k)$ and define $\mathcal{Q}^{(r)}\rightarrow \P^1$ to be the hyperbolic reduction of $\mathcal{Q}\rightarrow \P^1$ with respect to $\ell$ (see Section~\ref{subsection:hyperbolic-reduction}).
The hyperbolic reduction $\mathcal{Q}^{(r)}\rightarrow \P^1$ is itself a quadric fibration
and may be regarded as the relative Fano scheme of isotropic $(r+1)$-planes of $\mathcal{Q}\rightarrow \P^1$ containing $\ell$;
moreover, the $k$-birational equivalence class of $\mathcal{Q}^{(r)}$ does not depend on $\ell$ (see Section~\ref{sec:Q^(r)-construction} for these and additional properties). 
We prove:

\begin{thm}\label{maximallinearspace}
Over a field $k$ of characteristic $\neq 2$,  
fix $g\geq 2$,
and let $X$ be a smooth complete intersection of two quadrics in $\P^{2g+1}$.
Then
$F_{g-2}(X)(k)\neq \emptyset$ and
$\mathcal{Q}^{(g-2)}$ is $k$-rational
if and only if 
$F_{g-1}(X)(k)\neq \emptyset$.
\end{thm}
Here the hyperbolic reduction $\mathcal{Q}^{(g-2)}$ is a smooth threefold with a quadric surface fibration structure.

In every \emph{even} dimension, the analogous statement to Theorem~\ref{maximallinearspace} fails (see Section~\ref{intro-main3}).
The reason Theorem~\ref{maximallinearspace} gives a partial converse of Theorem~\ref{Fanoscheme} is the following birational description of the Fano schemes, which relates \(F_r(X)\) to the hyperbolic reduction \(\cal Q^{(r)}\):

\begin{thm}\label{thm:symmetric}
Over a field $k$ of characteristic $\neq 2$, let $X$ be a smooth complete intersection of two quadrics in $\P^N$.
Let $0\leq r\leq \lfloor\frac{N}{2}\rfloor -1$.
If $F_r(X)(k)\neq \emptyset$, then one of the following conditions holds.
\begin{enumerate}
\item $F_r(X)$ is $k$-birational to $\Sym^{r+1}\mathcal{Q}^{(r)}$.
\item $N=2g$ and $r=g-1$ for some $g\geq 1$. In this case,  
the subscheme of $F_{g-1}(X)$ parametrizing $(g-1)$-planes on $X$ 
disjoint from $\ell$ is $k$-isomorphic to
the subscheme of $\Sym^{g}\mathcal{Q}^{(g-1)}$ parametrizing $g$-tuples of distinct points of $\mathcal{Q}^{(g-1)}$, and they are \(0\)-dimensional schemes of length $\binom{2g+1}{g}$.
\end{enumerate}
\end{thm}
Two special cases of Theorem~\ref{thm:symmetric} were previously known. When
$N=2g+1 \geq 5$ is odd, $r=g-1$, and the field is \emph{algebraically closed},
Reid proved the birational equivalence of $F_{g-1}(X)$ and $\Sym^g C$, where \(C\) is the genus \(g\) hyperelliptic curve obtained as the Stein factorization of $F_{g}(\mathcal{Q}/\P^1)\rightarrow \P^1$ \cite[Section 4]{Reid-thesis}. The other previously known case is when $N>2$ and $r=0$: Colliot-Th\'el\`ene--Sansuc--Swinnerton-Dyer proved that $X$ is $k$-birational to $\mathcal{Q}^{(0)}$ \cite[Theorem 3.2]{CTSSD}.

Theorem~\ref{thm:symmetric} shows that Theorem~\ref{maximallinearspace} is a partial converse to Theorem~\ref{Fanoscheme} because symmetric powers of \(k\)-rational varieties are also \(k\)-rational \cite{Mattuck69}. However, it does not give the full converse because, in general, it is possible for a symmetric power of an irrational variety to be rational (see Remark~\ref{rem:sym-rational}) and thus $k$-rationality of $\mathcal{Q}^{(g-2)}$ may be stronger than that of $F_{g-2}(X)$. 
For Theorem~\ref{Fanoscheme}, we show that the statement follows from Theorem~\ref{thm:symmetric} by proving that a \(k\)-point on \(F_{r+1}(X)\) gives a section of the quadric fibration \(\cal Q^{(r)} \to \bb P^1\) and hence a \(k\)-rationality construction.

As another application of Theorem~\ref{thm:symmetric}, we prove the following (separable) $k$-unirationality criterion:
\begin{thm}\label{thm:fano-unirational}
Over a field $k$ of characteristic $\neq 2$, 
fix $N\geq 6$, and let $X$ be a smooth complete intersection of two quadrics in $\P^N$.
The following are equivalent:
\begin{enumerate}
\item\label{item:fano-unirational-1} $F_1(X)$ is separably $k$-unirational;
\item\label{item:fano-unirational-2} $F_1(X)$ is $k$-unirational;
\item\label{item:fano-unirational-3} $F_1(X)(k)\neq \emptyset$.
\end{enumerate}
In addition, if \(k=\bb R\) is the real numbers, then the above result holds for all Fano schemes of non-maximal linear subspaces. That is, for every $0\leq r\leq \lfloor \frac{N}{2}\rfloor -2$, \(F_r(X)\) is \(\bb R\)-unirational if and only if it has an \(\bb R\)-point.
\end{thm}

The bound on $N$ in Theorem~\ref{thm:fano-unirational} is 
crucial because $F_1(X)$ is never $\kbar$-unirational for $N \leq 5$. Our result extends previous results for $F_0(X)=X$ due to Manin \cite[Theorems 29.4 and 30.1]{manincubic86}, Knecht \cite[Theorem 2.1]{knecht15}, Colliot-Th\'el\`ene--Sansuc--Swinnerton-Dyer \cite[Remark 3.28.3]{CTSSD}, and Benoist--Wittenberg \cite[Theorem 4.8]{BW-IJ}.
Using Theorem~\ref{thm:symmetric}, the proof of Theorem~\ref{thm:fano-unirational} is reduced to showing separable \(k\)-unirationality of the hyperbolic reduction $\mathcal{Q}^{(1)}$ (and, if \(k=\bb R\), for \(\cal Q^{(r)}\); here we use a result of Koll\'ar over local fields \cite{{kollar1999}}).
Part of the difficulty in generalizing the result for higher \(r\) from \(\bb R\) to other fields lies in the discrepancy between \(k\)-points and \(0\)-cycles of degree \(1\).

In Section~\ref{section-specific-fields}, we also apply Theorems~\ref{Fanoscheme} and~\ref{thm:symmetric} to establish \(k\)-rationality results for \(F_r(X)\) for certain fields \(k\), extending earlier results that were previously known only for $X$.
More precisely, we prove results for $C_i$-fields, $p$-adic fields, totally imaginary number fields, and finite fields, generalizing results of Colliot-Th\'el\`ene--Sansuc--Swinnerton-Dyer \cite[Theorem 3.4]{CTSSD}.
Over algebraically closed fields, using work of Ramanan \cite{Ramanan81} relating Fano schemes of odd-dimensional \(X\) and moduli spaces of certain vector bundles on hyperelliptic curves, we also prove rationality results for these moduli spaces (Corollary~\ref{thm:modulirational}), partially extending the work of Newstead \cite{Newstead75, Newstead80} and King--Schofield \cite{KingSchofield}.

\subsection{Second maximal linear spaces on even-dimensional complete intersections over $\R$}\label{intro-main3}
In the latter part of the paper, we focus on rationality over the field \(\bb R\) of real numbers. For a smooth complete intersection of two quadrics defined over \(\bb R\), Theorem~\ref{Fanoscheme} implies that its Fano schemes of non-maximal linear spaces are $\bb C$-rational. One may further ask when these Fano schemes are rational over \(\bb R\).

The locus of real points encodes additional obstructions to rationality over \(\bb R\): if \(Y\) is an \(\bb R\)-rational smooth projective variety, then \(Y(\bb R)\) is necessarily connected and non-empty. In dimensions \(1\) and \(2\), this topological obstruction characterizes rationality for \(\bb C\)-rational varieties \cite{Comessatti1913}. In higher dimensions, however, this fails in general: in dimension \(\geq 3\), there are \(\bb C\)-rational varieties, with non-empty connected real loci, that are irrational over \(\bb R\) \cite[Theorem 5.7]{BW20}.
Among complete intersections of quadrics \(X \subset \bb P^5\), Hassett--Tschinkel showed that there exist examples that are irrational over \(\bb R\) despite \(X(\bb R)\) being non-empty and connected \cite{HT-intersection-quadrics}. (See also \cite{FJSVV, JJ} for other examples in dimension \(3\).)

Hassett--Koll\'ar--Tschinkel studied \(\bb R\)-rationality for \emph{even}-dimensional complete intersections of quadrics. In particular, they showed that a \(4\)-fold \(X \subset \bb P^6\) is \(\bb R\)-rational if and only if its real locus is non-empty and connected. We prove an analogous result for the second maximal linear spaces on \(X \subset \bb P^{2g}\):

\begin{thm}[Theorem~\ref{thm:fano-R-rationality-connected-real-locus-precise}]\label{thm:fano-R-rationality-connected-real-locus}
Over the real numbers, fix $g\geq 2$, and let $X$ be a smooth complete intersection of two quadrics in $\P^{2g}$.
Then $F_{g-2}(X)$ is $\R$-rational if and only if $F_{g-2}(X)(\R)$ is non-empty and connected. Furthermore, this is equivalent to \(\bb R\)-rationality of the surface $\mathcal{Q}^{(g-2)}$.
\end{thm}
Thus, in this case, the aforementioned necessary condition for \(\bb R\)-rationality is in fact sufficient. Moreover, we show that this property is determined by the real isotopy class of \(X\) (Section~\ref{sec:isotopy-consequences}). One can apply Theorem~\ref{thm:fano-R-rationality-connected-real-locus} to construct examples in any even dimension where $\mathcal{Q}^{(g-2)}$ is \(\bb R\)-rational but \(F_{g-1}(X)(\bb R)=\emptyset\) (see Corollary~\ref{cor:fano-R-rationality-connected-real-locus-isotopy}\eqref{item:even-maximallinearspace-fails}), contrasting the odd-dimensional case as shown by Theorem~\ref{maximallinearspace}. 
Furthermore, the analogue of Theorem~\ref{thm:fano-R-rationality-connected-real-locus} fails in every odd dimension (see Example~\ref{exmp:odd-Q^g-2-connected-irrational}).

As an application, in the case of \(4\)-dimensional \(X \subset \bb P^6\), combining Theorems~\ref{thm:fano-unirational} and~\ref{thm:fano-R-rationality-connected-real-locus} with earlier results in \cite{CTSSD,HassettKoll'arTschinkel}, we completely determine $\R$-rationality and $\R$-unirationality of the Fano schemes of non-maximal linear spaces on a $4$-fold complete intersection \(X\) of two real quadrics using an isotopy invariant that was studied by Krasnov \cite{krasnov-biquadrics} (see Section~\ref{sec:krasnov}). This invariant, which was used in the rationality classifications in \cite{HT-intersection-quadrics,HassettKoll'arTschinkel}, is defined using the signatures of the quadrics in the associated pencil \(\cal Q\to\bb P^1\).
In particular, we show that these \(\bb R\)-(uni)rationality properties of \(F_r(X)\) for \(r=0,1\) are controlled solely by the real isotopy class of \(X\), extending the results of \cite{HassettKoll'arTschinkel} on \(\bb R\)-rationality of \(X\):

\begin{cor}\label{cor:classification}
Over the real numbers, let $X$ be a smooth complete intersection of two quadrics in $\P^6$.
\begin{enumerate}
\item\label{item:4-fold-F_1-rational} $F_1(X)$ is $\R$-rational if and only if the Krasnov invariant is one of
$(1)$, $(3)$, $(1,1,1)$, $(2,2,1)$, $(1,1,1,1,1)$, $(2,1,2,1,1)$, or $(1,1,1,1,1,1,1)$.
\item\label{item:4-fold-F_1-unirational} $F_1(X)$ is $\R$-unirational
if and only if
the Krasnov invariant is one of those listed in~\ref{item:4-fold-F_1-rational}, or is
$(3,1,1)$, $(3,2,2)$, $(3,1,1,1,1)$, or $(2,2,1,1,1)$.
\item\label{item:4-fold-X-rational} \cite[Theorem 1.1]{HassettKoll'arTschinkel} $X$ is $\R$-rational if and only if the Krasnov invariant is one of those listed in~\ref{item:4-fold-F_1-unirational}, or is $(5)$, $(4,2,1)$, or $(3,3,1)$.
\item $X$ is $\R$-unirational
if and only if 
the Krasnov invariant is one of those listed in~\ref{item:4-fold-X-rational} or is $(5,1,1)$.
\end{enumerate}
\end{cor}

\subsection{Outline}
In Section~\ref{sec:background}, we recall preliminary results on Fano schemes of linear subspaces on complete intersections of quadrics, the definition and properties of hyperbolic reductions, lemmas on pencils of quadrics from Reid's thesis \cite{Reid-thesis}, and Benoist--Wittenberg's codimension 2 Chow scheme for threefolds. In Section~\ref{section-fano-hyperbolic}, we describe the hyperbolic reduction of a pencil of quadrics. We prove Theorems~\ref{Fanoscheme}, \ref{thm:symmetric}, and~\ref{thm:fano-unirational} in this section, and as consequences, we derive rationality results over certain fields in Section~\ref{section-specific-fields}. Section~\ref{section-maximal}, which is the most technical part of the paper, is devoted to the proof of Theorem~\ref{maximallinearspace} on odd-dimensional complete intersections. Finally, in Section~\ref{sec:real-even} we turn to the case of even-dimensional complete intersections over \(\bb R\). Here we prove Theorem~\ref{thm:fano-R-rationality-connected-real-locus} and Corollary~\ref{cor:classification}, and we give examples contrasting the behavior in the even- and odd-dimensional cases.

\subsection*{Notation}
Throughout \(k\) is a field of characteristic \(\neq 2\). For a variety \(X\), we use \(x\in X\) to denote a scheme-theoretic point.
We say a variety \(X\) over \(k\) is \(k\)-rational to emphasize that the rationality construction is defined over \(k\); when we say rational without specifying the ground field, we usually mean \(\kbar\)-rational over an algebraically closed field.

For a smooth variety $X$ over \(k\) and
\(i\geq 0\), we let \(\CH^i(X)\) denote the Chow group of codimension \(i\) cycles on \(X\). We denote the subgroup of algebraically trivial cycles by \(\CH^i(X)_{\alg} \subset \CH^i(X)\), and we denote the quotient by \(\NS^i(X) = \CH^i(X)/\CH^i(X)_{\alg}\).
For smooth projective varieties $X,Y, Z$ and correspondences $\Gamma\in \CH^i(X\times Y)$ and $\Gamma' \in \CH^j(Y\times Z)$, we denote their composition by
\[\Gamma'\circ\Gamma \coloneqq (\pi_{X\times Z})_*((\pi_{X\times Y})^*\Gamma\cdot (\pi_{Y\times Z})^*\Gamma')\in \CH^{i+j-\dim Y}(X\times Z),\]
where $\pi_{X\times Y}\colon X\times Y\times Z\rightarrow X\times Y$, $\pi_{Y\times Z}\colon X\times Y\times Z\rightarrow Y\times Z$, and $\pi_{X\times Z}\colon X\times Y\times Z\rightarrow X\times Z$ are the projections.
For a curve \(C\) over \(k\), we use \(\Pic(C)\) to denote the Picard group over \(k\), and \(\bPic_{C/k}\) to denote the relative Picard scheme over \(k\).
For a variety \(X\) over \(k\) and \(r\geq 0\), \(\Sym^r X\) denotes the \(r\)th symmetric power (over \(k\)).

If \(\ell_1, \ldots, \ell_{r} \subset \bb P^n\) are linear subspaces, we denote their span by \(\langle \ell_1, \ldots, \ell_r\rangle\).
For $n\geq 0$, we denote the $n\times n$ identity matrix by $I_n$.

\subsection*{Acknowledgements}
We thank Brendan Hassett and Jerry Wang for interesting discussions.
We thank Olivier Debarre for feedback
on an earlier version of the paper,
and Jean-Louis Colliot-Th\'el\`ene for comments, in particular for communicating to us the applications in Section~\ref{sec:arithmetic-applications} and for pointing out that Theorem~\ref{thm:clifford-algebra} does not require the assumption that \(F_{g-2}(X)(k)\neq\emptyset\).
We are also grateful to the anonymous referees for their careful reading of this paper and helpful comments.
This work started during a visit of the first author to UCLA, and she thanks Joaqu\'in Moraga and Burt Totaro for their hospitality and for providing a welcoming environment.

\section{Preliminaries}\label{sec:background}

\subsection{Fano schemes of $r$-planes on a smooth complete intersection of two quadrics}

Let \(X\subset\bb P^N\) be a smooth complete intersection of two quadrics. For non-negative integers $r\geq 0$, $F_r(X)$ denotes the Fano scheme of $r$-planes on $X$. In this section, we recall some preliminary results about these Fano schemes \(F_r(X)\).

\begin{lem}[{\cite[Lemmas A.3 and A.4]{CT-intersection-quadrics}}]\label{lem:Kuznetsov_fanoscheme}
Over a field \(k\) of characteristic \(\neq 2\), let $X$ be a smooth complete intersection of two quadrics in $\P^{N}$. The following hold.
\begin{enumerate}
\item 
If $r>\lfloor \frac{N}{2}\rfloor -1$,
$F_r(X)$ is empty.
\item 
If $0\leq r\leq \lfloor \frac{N}{2}\rfloor -1$,
$F_r(X)$ is non-empty, smooth, projective, and of dimension $(r+1)(N-2r-2)$.
\item If $0\leq r< \frac{N}{2} -1$ (or equivalently if $\dim F_r(X)>0$), then $F_r(X)$ is geometrically connected.
\end{enumerate}
\end{lem}

For the sake of completeness, we add the following results, which will not be used in the rest of the paper.

\begin{thm}
In the setting of Lemma~\ref{lem:Kuznetsov_fanoscheme}, the following hold.
\begin{enumerate}
\item\label{item:fanoschemes-structures-finite} If $N=2g$ and $r=g-1$ for some $g\geq 1$, $F_{g-1}(X)$ is a torsor under the finite group scheme $\bPic^0_{C/k}[2]$, where $C$ is a certain curve of genus $g$ associated to $X$. The curve $C$ depends on the choice of two quadrics defining $X$, but $\bPic^0_{C/k}[2]$ does not depend on this choice.
\item\label{item:fanoschemes-structures-torsor} If $N=2g+1$ and $r=g-1$ for some $g\geq 1$,
$F_{g-1}(X)$ is a torsor under $\bPic_{C/k}^0$, where $C$ is the curve of genus $g$ obtained as the Stein factorization of $F_g(\mathcal{Q}/\P^1)\rightarrow\P^1$.
\item\label{item:fanoschemes-structures-fano} If $0\leq r\leq \lfloor \frac{N}{2}\rfloor -2$, $F_r(X)$ is a Fano variety, i.e., the anti-canonical divisor $-K_{F_r(X)}$ is ample.
\end{enumerate}
\end{thm}
\begin{proof}
\eqref{item:fanoschemes-structures-finite}, \eqref{item:fanoschemes-structures-torsor}, and~\eqref{item:fanoschemes-structures-fano} respectively follow from 
\cite[Section 6]{BG14},
\cite[Theorem 1.1]{Wang-maximal-linear-spaces}, and \cite[Remarque 3.2]{DebarreManivel98}.
\end{proof}

\subsection{Quadric bundles and hyperbolic reductions}\label{subsection:hyperbolic-reduction}

In this section, we recall the definition and basic properties of quadric fibrations and of the hyperbolic reductions of a quadric fibrations; see, e.g., \cite[Section 2]{Kuznetsov-hyperbolic}, \cite[Section 1.3]{ABB14} for more details.

Let \(S\) be an integral separated Noetherian scheme over a field \(k\) of characteristic \(\neq 2\). A \defi{quadric fibration} over \(S\) is a morphism \(\cal Q\to S\) that can be written as a composition \(\cal Q\hookrightarrow\bb P_S(\cal E) \coloneqq \underline{\Proj}_S \Sym^\bullet (\cal E^\vee) \to S\)
where \(\cal E\) is a vector bundle and \(\cal Q\hookrightarrow\bb P_S(\cal E)\) is a divisor of relative degree 2 over \(S\). A quadric fibration is determined by a quadratic form \(q\colon\Sym^2\cal E\to\cal L^\vee\) with values in a line bundle \(\cal L\), and we define the \defi{degeneration divisor} of the quadric fibration to be the zero locus of the determinant of $q$ .

\begin{lem}[{\cite[Proposition 1.2.5]{ABB14}}]\label{lem:quadric-bundle-simple-degeneration}
    Let \(S\) be a smooth scheme over a field of characteristic \(\neq 2\), and let \(\pi\colon\cal Q\to S\) be a flat quadric fibration with smooth generic fiber. Then the degeneration divisor of \(\pi\) is smooth over \(k\) if and only if \(\cal Q\) is smooth over \(k\) and \(\pi\) has simple degeneration (i.e., the fibers of \(\pi\) have corank at most 1).
\end{lem}

A subbundle \(\cal F\subset \cal E\) is \defi{isotropic} if \(q|_{\cal F}=0\) (or, equivalently, if \(\bb P_S(\cal F)\subset\cal Q\)). An isotropic subbundle \(\cal F\subset\cal E\) is \defi{regular} if, for every (closed) point \(s\in S\), the fiber \(\bb P_S(\cal F)_s\) is contained in the smooth locus of the fiber \(\cal Q_s\).
If \(\cal F\) is regular isotropic, then \(\cal F\) is contained in the subbundle \[\cal F^\perp\coloneqq\Ker(\cal E\twoheadrightarrow\cal F^\vee\otimes\cal L^\vee)\] of \(\cal E\). Moreover, \(\cal F\) is in the kernel of the restriction \(q|_{\cal F^\perp}\), so we have an induced quadratic form on \(\cal F^\perp/\cal F\).

\begin{defn}
    The induced quadratic form \(\overline{q}\colon\Sym^2(\cal F^\perp/\cal F)\to\cal L^{\vee}\) is the \defi{hyperbolic reduction} of \(q\colon\Sym^2 \cal E\to\cal L^\vee\) with respect to the regular isotropic subbundle \(\cal F\). We also say that \(\overline{\cal Q}\coloneqq(\overline{q}=0)\) is the \defi{hyperbolic reduction} of \(\cal Q=(q=0)\) with respect to \(\bb P_S(\cal F)\).
\end{defn}

The process of hyperbolic reduction along a regular isotropic subbundle preserves the degeneration divisor of a quadric fibration:

\begin{lem}[{\cite[Corollary 1.3.9]{ABB14}}]\label{lem:degeneration-divisor-preserved}
Let $S$ be a smooth scheme over a field of characteristic $\neq 2$, and $\pi\colon \mathcal{Q}\rightarrow S$ be a quadric fibration with smooth generic fiber.
Let $\cal F\subset \cal E$ be a regular isotropic subbundle, and $\overline{\pi}\colon \overline{\cal Q}\rightarrow S$ be the hyperbolic reduction of $\mathcal{Q}$ respect to $\bb P_S(\cal F)$.
Then $\pi$ and $\overline{\pi}$ have the same degeneration divisor.
\end{lem}

Note that, more generally, hyperbolic reduction preserves the locus of corank \(\geq i\) fibers for each \(i\) \cite[Lemma 2.4]{KuznetsovShinder18}.

Hyperbolic reduction can be described geometrically in terms of the linear projection of \(\cal Q\subset\bb P_S(\cal E)\) from the linear subbundle \(\bb P_S(\cal F)\subset\cal Q\subset\bb P_S(\cal E)\) \cite[Proposition 2.5]{KuznetsovShinder18}.

\subsection{Lemmas of Reid on pencils of quadrics}

Next, over \emph{algebraically closed} fields, we recall several results proven by Reid \cite{Reid-thesis} that we will use in the proof of Theorem~\ref{thm:symmetric}.
In what follows, \defi{coordinate points} $p_0, \cdots, p_N\in \P^N$ mean a choice of coordinates for $\P^N$ such that the $N+1$ points given by intersecting coordinate hyperplanes are exactly $p_0,\cdots, p_N$.

\begin{lem}[{\cite[Lemma 2.2]{Reid-thesis}}]\label{lem:reid_1}
Over an algebraically closed field of characteristic $\neq 2$, let $X$ be a smooth complete intersection of two quadrics in $\P^N$, and let $\ell$ be an $r$-plane on $X$.
Then there exist coordinate points $p_0,\cdots, p_N\in \P^N$ such that $\langle p_0,\dots, p_r\rangle =\ell$ and $X$ is defined by two quadrics which correspond in these coordinates to symmetric matrices of the form
\[
\begin{pmatrix}
0 & I_{r+1} & 0\\
I_{r+1} & 0 & 0\\
0 & 0 & I_{N-2r-1}
\end{pmatrix}, \quad
\begin{pmatrix}
0 & M & *\\
M & 0 & *\\
* & * & *
\end{pmatrix},
\]
where $M$ is a diagonal \((r+1)\times(r+1)\) matrix with distinct diagonal entries.
\end{lem}

\begin{lem}[{\cite[Lemma 3.4]{Reid-thesis}}]\label{lem:reid_2}
Over an algebraically closed field of characteristic $\neq 2$, let $\ell, m$ be disjoint $r$-planes in $\P^{2r+1}$
and choose coordinate points $p_0,\dots, p_r, q_0,\dots, q_r\in \P^{2r+1}$ such that $\langle p_0,\dots,p_r\rangle=\ell$ and $\langle q_0,\dots, q_r\rangle =m$.
Furthermore, let $Q_1, Q_2$ be quadrics in $\P^{2r+1}$ which correspond in these coordinates to symmetric matrices of the form
\[
\begin{pmatrix}
    0 & I_{r+1}\\
    I_{r+1} & 0
\end{pmatrix}, \quad
\begin{pmatrix}
0 & M\\
M & 0
\end{pmatrix},
\]
where $M$ is a diagonal \((r+1)\times(r+1)\) matrix with distinct diagonal entries.
Then the set of $r$-planes on the singular complete intersection of quadrics $Y \coloneqq Q_1\cap Q_2$ coincides with
\[
\left\{\langle p_i, q_j\rangle_{i\in I, j\not\in I}\mid I\subset \left\{0,\dots, r\right\}
\right\}.
\]
\end{lem}

\begin{lem}\label{lem:reid-3}
Over an algebraically closed field of characteristic $\neq 2$, 
fix $g\geq 1$, and let $X$ be a smooth complete intersection of two quadrics in $\P^{2g}$.
Let $\ell, m$ be disjoint $(g-1)$-planes on $X$.
Then there exist coordinate points $p_0,\cdots, p_{2g}\in \P^{2g}$ such that $\langle p_0,\dots, p_{g-1}\rangle =\ell$, $\langle p_{g},\dots, p_{2g-1}\rangle=m$, and $X$ is defined by two quadrics which correspond in these coordinates to symmetric matrices of the form
\[
\begin{pmatrix}
0 & I_g & *\\
I_g & 0 & *\\
* & * & *
\end{pmatrix}, \quad
\begin{pmatrix}
0 & M & *\\
M & 0 & *\\
* & * & *
\end{pmatrix},
\]
where $M$ is a diagonal \(g\times g\) matrix with distinct diagonal entries.
\end{lem}

For the proof, we need the following result stated in \cite[page 44]{Reid-thesis}:

\begin{lem}\label{lem:elementary}
    Over an algebraically closed field of characteristic $\neq 2$, let $A, B$ be $n\times n$ matrices.
    Then the following conditions are equivalent:
    \begin{enumerate}
    \item\label{item:change-of-basis} there exists an invertible $2n\times 2n$ matrix $L$ of the form
    \[
    L = \begin{pmatrix}
    L_1 & 0\\
    0 & L_2
    \end{pmatrix}
    \]
    for some $n\times n$ matrices $L_1,L_2$,
    and there exists a diagonal $n\times n$ matrix $M$ with distinct diagonal entries such that
    \[
    L^T
    \begin{pmatrix}
    0 & A \\
    A^T & 0 \\ 
    \end{pmatrix}L =
    \begin{pmatrix}
    0 & I_n \\
    I_n & 0 \\ 
    \end{pmatrix}, \quad
        L^T
    \begin{pmatrix}
    0 & B \\
    B^T & 0 \\ 
    \end{pmatrix}L =
    \begin{pmatrix}
    0 & M \\
    M & 0 \\ 
    \end{pmatrix};
    \]
    \item\label{item:polynomial-separable} the polynomial $\det(\lambda A+ B)$ has $n$ distinct roots.
    \end{enumerate}
\end{lem}
\begin{proof}
\eqref{item:change-of-basis}$\Rightarrow$\eqref{item:polynomial-separable}
is immediate. 
As for \eqref{item:polynomial-separable}$\Rightarrow$\eqref{item:change-of-basis},
the assumption implies that $A$ is invertible and $BA^{-1}$ has $n$ distinct eigenvalues.
Let $C$ be an invertible $n\times n$ matrix such that $C^{-1}BA^{-1}C$ equals some diagonal $n\times n$ matrix $M$ with distinct diagonal entries.
Then we may take 
\[L \coloneqq \begin{pmatrix}(C^T)^{-1} & 0\\
0& A^{-1}C\end{pmatrix}.\]
\end{proof}

\begin{proof}[Proof of Lemma \ref{lem:reid-3}]
The proof is outlined in
\cite[page 44]{Reid-thesis}.
Here is a detailed argument.
Take coordinate points $p_0,\cdots, p_{2g}\in \P^{2g}$ such that $\langle p_0,\dots,p_{g-1}\rangle =\ell$ and $\langle p_g,\dots,p_{2g-1}\rangle =m$.
In these coordinates, any two quadrics defining $X$ correspond to symmetric matrices of the form
\[
S=
\begin{pmatrix}
0 & A & *\\
A^T & 0 & *\\
* & * & *
\end{pmatrix}, \quad
T=\begin{pmatrix}
0 & B & *\\
B^T & 0 & *\\
* & * & *
\end{pmatrix},
\]
where $A, B$ are $g\times g$ matrices,
and, by the smoothness of $X$, we may choose those quadrics so that $S$ is invertible.
By \cite[Corollary 3.7]{Reid-thesis}, the polynomial $\det(\lambda A+B)$ divides the polynomial $\det(\lambda S+ T)$,
where the latter has distinct $2g+1$ roots by the smoothness of $X$ \cite[Proposition 2.1]{Reid-thesis}.
Hence $\det(\lambda A+ B)$ has distinct $g$ roots, which implies that after a suitable coordinate change
we may take $A=I_g$ and $B$ to be diagonal with distinct diagonal entries by Lemma \ref{lem:elementary}.
Accordingly, we obtain coordinate points with the desired properties. 
\end{proof}

\subsection{$\bCH^2$-scheme of Benoist--Wittenberg}\label{sec:CH2-scheme} Throughout this section, let \(X\) be a smooth, proper, geometrically connected, \emph{geometrically rational} threefold over \(k\), and let $G_k \coloneqq \Gal(\overline{k}/k_p)$ be the absolute Galois group of the perfect closure $k_p$ of $k$. We recall several key properties of Benoist--Wittenberg's codimension 2 Chow scheme of \(X\) \cite{BW-IJ}, which we will use in the proof of Theorem~\ref{maximallinearspace}.

For such a threefold \(X\), Benoist and Wittenberg use K-theory to define a functor \(\CH^2_{X/k,\fppf}\) for codimension 2 cycles that is analogous to the Picard functor \(\Pic_{X/k,\fppf}\). They show that this functor is represented by a smooth group scheme \(\bCH^2_{X/k}\) over \(k\) with the following properties \cite[Theorem 3.1]{BW-IJ}:
\begin{enumerate}
    \item The identity component \((\bCH^2_{X/k})^0\) is an abelian variety, which we refer to as the \defi{intermediate Jacobian} of \(X\). Its base change to the perfect closure of \(k\) agrees with the intermediate Jacobian as defined by Murre \cite{Murre85,ACMV-descend}.
    \item There is a \(G_k\)-equivariant isomorphism \(\CH^2(X_{\kbar})\cong\bCH^2_{X/k}(\kbar)\).
    \item The component group \(\bCH^2_{X/k}/(\bCH^2_{X/k})^0\) is identified with the \(G_k\)-module \(\NS^2(X_\kbar)\).
    \item\label{item:classic-IJ-obstruction} If \(X\) is \(k\)-rational, then there is a smooth projective (not necessarily connected) curve \(B\) over \(k\) such that \(\bCH^2_{X/k}\) is a principally polarized direct factor of \(\bPic_{B/k}\).
\end{enumerate}
For each class \(\gamma\in\NS^2(X_{\kbar})^{G_k}=(\bCH^2_{X/k}/(\bCH^2_{X/k})^0)(k)\), its inverse image \((\bCH^2_{X/k})^\gamma\) in \(\bCH^2_{X/k}\) is a(n \'etale) \((\bCH^2_{X/k})^0\)-torsor. The quotient map is a group homomorphism, so we have an equality of \((\bCH^2_{X/k})^0\)-torsors \[[(\bCH^2_{X/k})^\gamma] + [(\bCH^2_{X/k})^{\gamma'}] = [(\bCH^2_{X/k})^{\gamma+\gamma'}]\] for any \(\gamma,\gamma'\in\NS^2(X_{\kbar})^{G_k}\).

Hassett--Tschinkel, over \(k=\bb R\) \cite{HT-intersection-quadrics}, and Benoist--Wittenberg, over arbitary fields \cite{BW-IJ}, (see also \cite{HT-cycle} for \(k\subset\bb C\)) observed that these intermediate Jacobian torsors can be used to refine the rationality obstruction~\eqref{item:classic-IJ-obstruction}. We call their refined obstruction the \defi{intermediate Jacobian torsor obstruction} to rationality. For simplicity, we state a special case of this obstruction, which is enough for our application:

\begin{thm}[{Special case of \cite[Theorem 3.11]{BW-IJ}}]\label{thm-BWIJ}
    For \(X\) as above, assume that there exists an isomorphism \((\bCH^2_{X/k})^0\cong\bPic^0_{C/k}\) of principally polarized abelian varieties for some smooth, projective, geometrically connected curve \(C\) of genus \(\geq 2\). If \(X\) is \(k\)-rational, then for every \(\gamma\in(\NS^2 X_{\kbar})^{G_k}\) there exists an integer \(d\) such that \((\bCH^2_{X/k})^\gamma\) and \(\bPic^d_{C/k}\) are isomorphic as \(\bPic^0_{C/k}\)-torsors.
\end{thm}

\section{Fano schemes of linear spaces and hyperbolic reductions of pencils of quadrics}\label{section-fano-hyperbolic}

In this section, we construct the hyperbolic reduction of a pencil of quadrics \(\cal Q \to \bb P^1\) with respect to a linear subspace in the base locus \(X\). We show several properties about the hyperbolic reductions, relating them to linear spaces on \(X\), and we prove Theorems~\ref{Fanoscheme}, \ref{thm:symmetric}, and~\ref{thm:fano-unirational}.

\subsection{Construction of $\phi^{(r)}\colon\mathcal{Q}^{(r)}\rightarrow\P^1$ and its properties}\label{sec:Q^(r)-construction}

Throughout Section~\ref{sec:Q^(r)-construction}, we will work in the following setting.
Fix an arbitrary field $k$ of characteristic $\neq 2$ and an integer $N\geq 2$. Let $X$ be a smooth complete intersection of two quadrics in $\P^N$, and let $\phi\colon \mathcal{Q}\rightarrow \P^1$ be the associated pencil of quadrics.
Let $0\leq r\leq \lfloor\frac{N}{2}\rfloor -1$.
By Lemma~\ref{lem:Kuznetsov_fanoscheme}, the Fano scheme $F_r(X)$ of $r$-planes on $X$ is non-empty. 
We will assume $F_r(X)(k)\neq \emptyset$ in what follows.

Choose
$\ell\in F_r(X)(k)$. The projection $\pi_{\ell}\colon\bb P^N \dashrightarrow \bb P^{N-r-1}$ away from $\ell$ induces diagrams
\[
\begin{tikzcd}[column sep=tiny]
&\P_{\P^{N-r-1}}(\O^{\oplus r+1}\oplus \O(1))\arrow[ld, "\bl_{\ell}"'] \arrow[rd] & \\
\P^{N} \arrow[rr, "\pi_{\ell}", dashed]& & \P^{N-r-1} ,
\end{tikzcd}
\;
\begin{tikzcd}
& \widetilde{X} \arrow[ld, "\bl_{\ell}"'] \arrow[rd, "f"] & \\
X \arrow[rr, "\pi_{\ell}|_X", dashed]& & \P^{N-r-1}           ,
\end{tikzcd}
\;
\begin{tikzcd}[column sep=scriptsize]
& \widetilde{\mathcal{Q}} \arrow[ld, "\bl_{\P^1\times \ell}"'] \arrow[rd, "h"] & \\
\mathcal{Q} \arrow[rr, "\id\times \pi_{\ell}|_{\mathcal{Q}}", dashed]& & \P^1\times \P^{N-r-1}           .
\end{tikzcd}
\]

To be more explicit, choose homogeneous coordinates $x_0,\dots, x_{N}$ for $\P^{N}$ so that
\[\ell=\left\{x_{r+1}=\dots= x_{N}=0\right\}.
\]
Then there exist forms $l_{ij}, q_i\in k[x_{r+1}, \dots, x_{N}]$ with $\deg l_{ij}=1, \deg q_i =2$ such that
\begin{align*}
X
=\left\{
\begin{pmatrix}
l_{00} & \dots &l_{0r} & q_0\\
l_{10} & \dots & l_{1r} & q_1
\end{pmatrix}
\begin{pmatrix}
x_0\\
\vdots\\
x_r\\
1
\end{pmatrix}
=0
\right\}
\subset \P^{N}.
\end{align*}
Choosing a section $z$ whose zero set equals the divisor $\P_{\P^{N-r-1}}(\O^{\oplus r+1})\subset \P_{\P^{N-r-1}}(\O^{\oplus r+1 }\oplus \O(1))$,
there exist homogeneous coordinates $y_{r+1},\dots, y_{N}$ for $\P^{N-r-1}$ such that $x_l = y_l z$,
and
\[
\widetilde{X}=
\left\{
\begin{pmatrix}
l_{00} & \dots &l_{0r} & q_0\\
l_{10} & \dots & l_{1r} & q_1
\end{pmatrix}
\begin{pmatrix}
x_0\\
\vdots\\
x_r\\
z
\end{pmatrix}
=0
\right\}\subset \P_{\P^{N-r-1}}(\O^{r+1}\oplus \O(1)),
\]
where $l_{ij}, q_i$ are in $y_{r+1},\dots, y_{N}$.
We also have
\begin{align*}
&\mathcal{Q}
=\left\{
\begin{pmatrix}
s & t
\end{pmatrix}
\begin{pmatrix}
l_{00} & \dots &l_{0r} & q_0\\
l_{10} & \dots & l_{1r} & q_1
\end{pmatrix}
\begin{pmatrix}
x_0\\
\vdots\\
x_r\\
1
\end{pmatrix}
=0
\right\}
\subset \P^1\times \P^{N},\\
&\widetilde{\mathcal{Q}}=
\left\{
\begin{pmatrix}
s & t
\end{pmatrix}
\begin{pmatrix}
l_{00} & \dots &l_{0r} & q_0\\
l_{10} & \dots & l_{1r} & q_1
\end{pmatrix}
\begin{pmatrix}
x_0\\
\vdots\\
x_r\\
z
\end{pmatrix}
=0
\right\}\subset \P^1\times \P_{\P^{N-r-1}}(\O^{\oplus r+1}\oplus \O(1)).
\end{align*}

We now define
\begin{align*}
&\mathcal{P}^{(r)}\coloneqq
\left\{
\begin{pmatrix}
s & t
\end{pmatrix}
\begin{pmatrix}
l_{00} & \dots &l_{0r}\\
l_{10} & \dots & l_{1r}
\end{pmatrix}
=0
\right\}
\subset \P^1\times \P^{N-r-1}, \\
&\mathcal{Q}^{(r)}
\coloneqq
\left\{
\begin{pmatrix}
s & t
\end{pmatrix}
\begin{pmatrix}
l_{00} & \dots &l_{0r} & q_0\\
l_{10} & \dots & l_{1r} & q_1
\end{pmatrix}
=0
\right\}
\subset \P^1\times \P^{N-r-1}.
\end{align*}
The first projection defines a $\P^{N-2r-2}$-bundle $\mathcal{P}^{(r)}\rightarrow \P^1$, which restricts to a morphism $\phi^{(r)}\colon \mathcal{Q}^{(r)}\rightarrow \P^1$, which is a quadric fibration if \(\dim\mathcal Q^{(r)} \geq 1\).
(By convention, $\mathcal{Q}^{(-1)}=\mathcal{Q}$ and $\phi^{(-1)}=\phi$.)
Furthermore, define \[E^{(r)}\coloneqq h^{-1}(\mathcal{Q}^{(r)})=\P_{\mathcal{Q}^{(r)}}(\O^{\oplus r+1}\oplus \O(0,1)),\]
and let $\pi^{(r)}\colon E^{(r)}\rightarrow \mathcal{Q}^{(r)}$ be the projection.

\begin{lem}\label{lem:hyperbolic-reduction}
$\mathcal{Q}^{(r)}$ satisfies the following properties.
\begin{enumerate}
\item\label{item:hyperbolic-reduction-pencil} $\phi^{(r)}\colon\mathcal{Q}^{(r)}\rightarrow \P^1$
is the hyperbolic reduction of $\phi\colon \mathcal{Q}\rightarrow \P^1$ with respect to $\ell$.
In particular, the degeneracy locus of $\phi^{(r)}$ is defined by a separable polynomial of degree $N+1$.
\item\label{item:hyperbolic-reduction-smooth} $\mathcal{Q}^{(r)}$ is smooth of dimension $N-2r-2$. 
If $N-2r-2>0$, $\mathcal{Q}^{(r)}$ is geometrically connected.
\item\label{item:hyperbolic-reduction-embedding} $\pi^{(r)}\colon E^{(r)}\rightarrow \mathcal{Q}^{(r)}$ induces an embedding $\mathcal{Q}^{(r)}\hookrightarrow F_{r+1}(\mathcal{Q}/\P^1)$ over $\P^1$,
whose image is the relative Fano
scheme of isotropic $(r+1)$-planes
of
$\phi\colon \mathcal{Q}\rightarrow \P^1$
containing $\ell$. 
\item\label{item:hyperbolic-reduction-image}
The restriction
$\pi^{(r)}|_{E^{(r)}\cap (\P^1\times \widetilde{X})}\colon E^{(r)}\cap (\P^1\times \widetilde{X})\rightarrow  \mathcal{Q}^{(r)}$ induces an isomorphism
\[
\mathcal{Q}^{(r)}\setminus\left\{l_{ij} = 0 \, (i= 0,1, j=1,\dots, r)\right\}\xrightarrow{\sim} \left\{m\in F_r(X)\mid \dim \ell\cap m = r-1 \text{ and }\langle \ell,m\rangle \not\subset X\right\}.
\]
\item\label{item:hyperbolic-reduction-independence} 
If $N-2r-2>0$, the $k$-birational equivalence class of $\mathcal{Q}^{(r)}$
does not depend on $\ell$.
\item\label{item:hyperbolic-reduction-max-min} If $N>2$, $\mathcal{Q}^{(0)}$ is $k$-birational to $X$.
If $N=2g$ for some $g\geq 1$, $\mathcal{Q}^{(g-1)}$ is of dimension zero and length $2g+1$.
If $N=2g+1$ for some $g\geq 1$, $\mathcal{Q}^{(g-1)}$ is $k$-isomorphic to the curve $C$ of genus $g$ obtained as the Stein factorization of $F_{g}(\mathcal{Q}/\P^1)\rightarrow \P^1$.
\item\label{item:index}
If $N-2r-2>0$, $\mathcal{Q}^{(r)}$ has a $0$-cycle of degree $1$,
in other words, the index of $\mathcal{Q}^{(r)}$ is $1$.
\item\label{item:rationalpoint}
If $N-3r-3\geq 0$, $\mathcal{Q}^{(r)}$ has a $k$-point.
\end{enumerate}
\end{lem}
\begin{proof}
\eqref{item:hyperbolic-reduction-pencil}:
For $x =[s:t]\in \P^1$, $\mathcal{Q}_{\kappa(x)}$ corresponds to the symmetric matrix
\[
\begin{pmatrix}
0 & A\\
A^T & B
\end{pmatrix},
\]
where $A$ is the Jacobian matrix for $\{s l_{00} + t l_{10} = \dots = s l_{0r} + t l_{1r} = 0\}$ and $B$ is the Hessian matrix for $s q_0 + t q_1$.
This in turn shows that 
\[\ell^{\perp}=
\left\{
\begin{pmatrix}
s & t
\end{pmatrix}
\begin{pmatrix}
l_{00} & \dots &l_{0r}\\
l_{10} & \dots & l_{1r}
\end{pmatrix}
=0
\right\}
\subset \P^{N}_{\kappa(x)}\]
and the hyperbolic reduction of $\mathcal{Q}_{\kappa(x)}$ with respect to $\ell$ equals $\mathcal{Q}^{(r)}_{\kappa(x)} \subset \mathcal{P}^{(r)}_{\kappa(x)}=\ell^\perp/\ell \subset \P^{N-r-1}_{\kappa(x)}=\P^{N}_{\kappa(x)}/\ell$.
The first statement follows.
As for the second statement, Lemma~\ref{lem:degeneration-divisor-preserved} shows that the degeneracy locus of $\phi^{(r)}\colon \mathcal{Q}^{(r)}\rightarrow \P^1$ is the same as that of $\phi\colon \mathcal{Q}\rightarrow \P^1$, and the latter is defined by a separable polynomial of degree $N+1$ by \cite[Proposition 2.1]{Reid-thesis}.

\eqref{item:hyperbolic-reduction-smooth}:
By \eqref{item:hyperbolic-reduction-pencil} and Lemma~\ref{lem:quadric-bundle-simple-degeneration},
$\mathcal{Q}^{(r)}$ is smooth.
Moreover, we deduce from \eqref{item:hyperbolic-reduction-pencil} that we have $\dim \mathcal{Q}^{(r)}=N-2r-2$.
This implies that $\mathcal{Q}^{(r)}$ is a complete intersection of $r+2$ ample divisors in $\P^1\times \P^{N-r-1}$.
So if $N-2r-2>0$, then $\mathcal{Q}^{(r)}$ is geometrically connected.

\eqref{item:hyperbolic-reduction-embedding}, \eqref{item:hyperbolic-reduction-image}:
Using the equations for $\mathcal{Q}$ and $\mathcal{Q}^{(r)}$,
we may observe that
$\pi^{(r)}\colon E^{(r)}\rightarrow \mathcal{Q}^{(r)}$ is a $\P^{r+1}$-bundle 
whose fibers are
the $(r+1)$-planes of $\phi\colon \mathcal{Q}\rightarrow \P^1$
containing $\ell$.
The induced morphism $\mathcal{Q}^{(r)}\rightarrow F_{r+1}(\mathcal{Q}/\P^1)$ over $\P^1$ is an isomorphism onto its image
because the blow-up of $\mathcal{Q}$ along $\P^1\times \ell$
yields the inverse map.

Similarly, using the equations for $X$ and $\mathcal{Q}^{(r)}$, 
it is direct to see that 
\[X\cap \left\{l_{ij}=0\, (i=0,1, j=0,\cdots, r)\right\}=\left\{l_{ij}=q_i = 0 \, (i=0,1, j=0,\cdots, r)\right\}\subset \P^N\] 
is the union of all linear spaces on $X$ containing $\ell$,
and that 
over the complement of the proper closed subscheme $\mathcal{Q}^{(r)}\cap\left\{l_{ij} = 0 \, (i= 0,1, j=1,\dots, r)\right\}\subset \mathcal{Q}^{(r)}$, the morphism $\pi^{(r)}|_{E^{(r)}\cap (\P^1\times \widetilde{X})}\colon E^{(r)}\cap (\P^1\times \widetilde{X})\rightarrow \mathcal{Q}^{(r)}$ is
a $\P^{r}$-bundle whose fibers are the $r$-planes $m$ on $X$ such that $\dim \ell\cap m = r-1$ and $\langle \ell, m\rangle \not\subset X$.
The induced morphism 
\[\mathcal{Q}^{(r)}\setminus\left\{l_{ij} = 0 \, (i= 0,1, j=1,\dots, r)\right\}\rightarrow\left\{m\in F_r(X)\mid \dim \ell\cap m = r-1 \text{ and }\langle \ell,m\rangle \not\subset X\right\}\] 
is an isomorphism because the blow-up of $X$ along $\ell$ yields the inverse map.

\eqref{item:hyperbolic-reduction-independence}:
It is enough for us to show that the $k(\P^1)$-isomorphism class of the generic fiber of $\mathcal{Q}^{(r)}\rightarrow \P^1$ does not depend on $\ell$.
For $\ell, \ell'\in F_r(X)(k)$, denote by $\mathcal{Q}^{(r)}_\ell\rightarrow \P^1, \mathcal{Q}^{(r)}_{\ell'}\rightarrow \P^1$
the corresponding hyperbolic reductions. Letting $H$ denote a split quadric, we have
\[
Q_{k(\P^1)}\simeq (\mathcal{Q}^{(r)}_{\ell})_{k(\P^1)}\perp H \simeq (\mathcal{Q}^{(r)}_{\ell'})_{k(\P^1)}\perp H.
\]
The Witt cancellation theorem \cite[Theorem 8.4]{EKM-quadratic-forms}
then shows $(\mathcal{Q}^{(r)}_\ell)_{k(\P^1)}\simeq (\mathcal{Q}^{(r)}_{\ell'})_{k(\P^1)}$.

\eqref{item:hyperbolic-reduction-max-min}:
If $N>2$, then $\mathcal{Q}^{(0)}$ is $k$-birational to the image of $X$ under the projection away from a $k$-point on $X$, hence $k$-birational to $X$.
If $N=2g$ for some $g\geq 1$, $\mathcal{Q}^{(g-1)}$ is a complete intersection of $g$ divisors of type $(1,1)$ and a divisor of type $(1,2)$ in $\P^1\times \P^{g}$, hence $\mathcal{Q}^{(g-1)}$ is of dimension zero and length $2g+1$.
If $N=2g+1$ for some $g\geq 1$,
the composition $\mathcal{Q}^{(g-1)}\hookrightarrow F_{g}(\mathcal{Q}/\P^1)\rightarrow C$ is an isomorphism over \(\P^1\)
since the double covers $\phi^{(g-1)}\colon\mathcal{Q}^{(g-1)}\rightarrow \P^1$ and $C\rightarrow \P^1$ are both branched over the degeneracy locus of $\phi\colon \mathcal{Q}\rightarrow \P^1$.

\eqref{item:index}: 
Since $\mathcal{Q}^{(r)}$ is a complete intersection of $r+1$ divisors of type $(1,1)$ and a divisor of type $(1,2)$ in $\P^1\times \P^{N-r-1}$, 
the zero-cycle $\mathcal{Q}^{(r)}\cdot (\P^1\times \P^{r+1}) - (r+1) \mathcal{Q}^{(r)}\cdot (*\times \P^{r+2})$ is of degree $1$.

\eqref{item:rationalpoint}:
If $N-3r-3\geq 0$, 
then $\left\{l_{ij}=0\, (i=0,1, j=1,\cdots, r)\right\} \subset\P^{N-r-1}$
is non-empty by a dimension count.
Using the equations for $\mathcal{Q}^{(r)}$
and the fact that $l_{ij}$ are linear forms, $\mathcal{Q}^{(r)}\cap\left\{l_{ij}=0\, (i=0,1, j=0,\cdots, r)\right\}$ has a $k$-point.
This concludes the proof.
\end{proof}

The following result is well-known to the experts; see \cite[Remark 3.4.1(a)]{CTSSD} and \cite[Section 6.1]{HT-equivariant}, where a special case
is stated.

\begin{prop}\label{HT}
The following conditions are equivalent:
\begin{enumerate}
\item\label{item:HT-Fano-scheme-pt} $F_{r+1}(X)(k)\neq \emptyset$;
\item\label{item:HT-rel-plane} $\phi\colon \mathcal{Q}\rightarrow \P^1$ has 
a relative 
$(r+1)$-plane containing $\P^1\times \ell$;
\item\label{item:HT-section} $\phi^{(r)}\colon\mathcal{Q}^{(r)}\rightarrow \P^1$
has a section.
\end{enumerate}
If any of these equivalent conditions holds, 
then $\mathcal{Q}^{(r)}$ is $k$-rational.
\end{prop}
\begin{proof}
For \eqref{item:HT-Fano-scheme-pt}$\Leftrightarrow$\eqref{item:HT-rel-plane}, the Amer--Brumer theorem \cite[Theorem 2.2]{Leep} shows that $F_{r+1}(X)(k)\neq\emptyset$ if and only if $F_{r+1}(\mathcal{Q}_{k(\P^1)})(k(\P^1))\neq\emptyset$;
by the Witt extension theorem \cite[Theorem 8.3]{EKM-quadratic-forms}, 
the latter is equivalent to the existence of an $(r+1)$-plane
on $\mathcal{Q}_{k(\P^1)}$, containing $\ell$, and defined over $k(\P^1)$.
Next, the equivalence \eqref{item:HT-rel-plane}$\Leftrightarrow$\eqref{item:HT-section} 
follows from Lemma~\ref{lem:hyperbolic-reduction}\eqref{item:hyperbolic-reduction-embedding}.
Finally, if~\eqref{item:HT-section} holds, then the generic fiber of \(\phi^{(r)}\) is a smooth quadric with a \(k(\P^1)\)-point, so \(\cal Q^{(r)}\) is \(k\)-rational.
\end{proof}

\begin{rem}
If we allow $X$ to be singular and assume that $\ell\in F_r(X)(k)$ is entirely contained in the smooth locus of $X$, most of the results in this section, after some suitable fixes, still hold.
An analogue of Proposition~\ref{HT} is also true. The point is that if $\mathcal{Q}_{k(\P^1)}$ is singular, then $\mathcal{Q}^{(r)}_{k(\P^1)}$ is singular, but any singular quadric contains a rational point.
We do not state and prove results in this generality in this paper because it is not necessary for our purposes.
\end{rem}

\subsection{Proofs of Theorems~\ref{Fanoscheme}, \ref{thm:symmetric}, and~\ref{thm:fano-unirational}}

\begin{proof}[Proof of Theorem~\ref{thm:symmetric}]
Under the assumptions of Theorem~\ref{thm:symmetric}, choose $\ell\in F_r(X)(k)$ so that $\mathcal{Q}^{(r)}$ is defined.
For $m\in F_r(X)$, not necessarily a $k$-point, we say that $m$ is \defi{general with respect to $\ell$} if:
(1) $\ell\cap m=\emptyset$; and (2) $X\cap \langle \ell, m\rangle \subset \langle \ell, m\rangle = \P^{2r+1}$ is defined by two quadrics which, over an algebraic closure of the residue field of $m$ and for some choice of coordinates, correspond to symmetric matrices of the form
\[
\begin{pmatrix}
    0 & I_{r+1}\\
    I_{r+1} & 0
\end{pmatrix}, \quad
\begin{pmatrix}
0 & M\\
M & 0
\end{pmatrix},
\]
where $M$ is diagonal with distinct diagonal entries.
By Lemma \ref{lem:elementary},
property (2) is equivalent to separability of a certain polynomial associated to $\ell$ and $m$. Define
\[
U\coloneqq\left\{m \in F_r(X) \mid m \text{ is general with respect to }\ell \right\}.
\]
$U$ is a priori only defined over $\kbar$, but it actually descends to $k$ as $\ell$ is defined over $k$.
Clearly, $U$ is open in $F_r(X)$, and moreover, it is non-empty by Lemma~\ref{lem:reid_1}.
For every $m\in U$, Lemma~\ref{lem:reid_2} implies that the (singular) intersection $X\cap \langle \ell, m\rangle \subset \langle \ell,m\rangle = \bb P^{2r+1}$ contains exactly $r+1$ distinct $r$-planes $m_1,\dots, m_{r+1}$ that intersect $\ell$ along an $(r-1)$-plane.
Using the isomorphism in Lemma \ref{lem:hyperbolic-reduction}\eqref{item:hyperbolic-reduction-image},
we now define a morphism
\[ 
\psi\colon U\rightarrow \Sym^{r+1}\mathcal{Q}^{(r)},\quad
m\mapsto (m_1,\dots, m_{r+1}).
\]
The morphism $\psi$ is defined over $k$ by the same reasoning as above.
Moreover, $\psi$ is one-to-one onto its image on the level of $\kbar$-points
because we have $\langle m_1,\cdots, m_{r+1}\rangle =\langle \ell,m\rangle$ and $m$ is the unique $r$-plane on $X\cap \langle \ell,m\rangle$ that is disjoint from $\ell$.
Since $\dim F_r(X) = \dim \Sym^{r+1}\mathcal{Q}^{(r)}=(r+1)(N-2r-2)$ by Lemma \ref{lem:hyperbolic-reduction}\eqref{item:hyperbolic-reduction-smooth}
and $\psi$ factors through the smooth locus of $\Sym^{r+1}\mathcal{Q}^{(r)}$,
it then follows that 
$\psi$ is an open immersion.
If $N-2r-2>0$, then $F_r(X)$ and $\Sym^{r+1}\mathcal{Q}^{(r)}$ are both geometrically integral by Lemmas~\ref{lem:Kuznetsov_fanoscheme} and~\ref{lem:hyperbolic-reduction}, so $F_r(X)$ is $k$-birational to $\Sym^{r+1}\mathcal{Q}^{(r)}$.

It remains to consider the case $N-2r-2=0$. Set $g=r+1$.
We have $\dim F_{g-1}(X)=\dim \Sym^{g}\mathcal{Q}^{(g-1)}=0$.
In this case, Lemma~\ref{lem:reid-3} shows that if $m\in F_{g-1}(X)$ satisfies property (1) of being general with respect to \(\ell\), then property (2) is automatic.
Hence $U$ coincides with the subscheme of $F_{g-1}(X)$ parametrizing $(g-1)$-planes on $X$ disjoint from $\ell$.
It remains to show that $\psi$ gives an isomorphism between $U$ and the subscheme of $\Sym^{g}\mathcal{Q}^{(g-1)}$ parametrizing $g$-tuples of distinct points of $\mathcal{Q}^{(g-1)}$.
The former is of length $\binom{2g+1}{g}$ by the analysis on the configuration of the $(g-1)$-planes on $X$ due to Reid \cite[Theorem 3.8]{Reid-thesis},
and the latter is also of length $\binom{2g+1}{g}$
by Lemma~\ref{lem:hyperbolic-reduction}\eqref{item:hyperbolic-reduction-max-min}. The claim thus follows, and this concludes the proof of the theorem.
\end{proof}

\begin{proof}[Proof of Theorem~\ref{Fanoscheme}]
If $F_{r+1}(X)(k)\neq \emptyset$, then $F_r(X)(k)\neq \emptyset$, and hence $\mathcal{Q}^{(r)}$ is defined.
Proposition~\ref{HT} shows that $\mathcal{Q}^{(r)}$ is $k$-rational, and
Theorem~\ref{thm:symmetric}
shows that $F_r(X)$ is $k$-birational to 
$\Sym^{r+1}\mathcal{Q}^{(r)}$, hence to $\Sym^{r+1}\P^{N-2r-2}$.
Symmetric powers of a projective space over $k$ are $k$-rational by a result of Mattuck \cite{Mattuck69}, completing the proof.
\end{proof}

\begin{proof}[Proof of Theorem~\ref{thm:fano-unirational}]
The implications \eqref{item:fano-unirational-1}\(\Rightarrow\)\eqref{item:fano-unirational-2}\(\Rightarrow\)\eqref{item:fano-unirational-3} are standard, so it remains to show \eqref{item:fano-unirational-3}\(\Rightarrow\)\eqref{item:fano-unirational-1}.

First we address the case of \(F_1(X)\) and arbitrary \(k\).
For this, by Theorem~\ref{thm:symmetric} and \cite{Mattuck69},
it is enough to show that 
if $F_1(X)(k)\neq\emptyset$ then $\mathcal{Q}^{(1)}$ is separably $k$-unirational.
We show this statement by induction on $N$.
By \cite[Lemma 4.9]{BW-IJ}, we may and will assume that $k$ is infinite. 

As for the base case $N=6$, by Lemma~\ref{lem:hyperbolic-reduction}\eqref{item:rationalpoint}, $\mathcal{Q}^{(1)}$ has a $k$-point.
Since $\phi^{(1)}\colon \mathcal{Q}\rightarrow \P^1$ is a conic bundle with $7$ singular fibers by Lemma~\ref{lem:hyperbolic-reduction}\eqref{item:hyperbolic-reduction-pencil} and \eqref{item:hyperbolic-reduction-smooth}, 
a result of Koll\'ar--Mella \cite[Theorem 7]{kollarmella17} shows that $\mathcal{Q}^{(1)}$ is $k$-unirational with a degree $8$ dominant rational map $\P^2\dashrightarrow \mathcal{Q}^{(1)}$.
Since $k$ is of characteristic $\neq 2$, $\mathcal{Q}^{(1)}$ is separably $k$-unirational.

For $N>6$, choose $\ell\in F_1(X)(k)$ and consider a general pencil $X\dashrightarrow \P^1$ of hyperplane sections containing $\ell$.
By the proof of \cite[Lemma A.4]{CT-intersection-quadrics},
a general member of the pencil is smooth,
hence 
the blow-up along the base locus
$X'\rightarrow \P^1$ has a smooth generic fiber.
Let $\mathcal{Q}^{(1)}\dashrightarrow \P^1$ be the induced pencil. 
For the blow-up along the base locus
$(\mathcal{Q}^{(1)})'\rightarrow \P^1$, the generic fiber is the hyperbolic reduction of the generic fiber of $X'\rightarrow \P^1$
with respect to the line $\ell$, and it is separably $k(\P^1)$-unirational by induction.
This implies that $(\mathcal{Q}^{(1)})'$ is separably $k$-unirational, so is $\mathcal{Q}^{(1)}$. 
This completes the proof for \(F_1(X)\).

It remains to show for \(0\leq r\leq \lfloor \frac{N}{2}\rfloor -2\) that if \(F_r(X)(\bb R)\neq\emptyset\), then \(F_r(X)\) is \(\bb R\)-unirational.
Choose $\ell\in F_r(X)(\R)$ so that $\mathcal{Q}^{(r)}$ is defined.
Then $\mathcal{Q}^{(r)}$ has an $\R$-point by Lemma~\ref{lem:hyperbolic-reduction}\eqref{item:index}, so it is $\R$-unirational by \cite[Collorary 1.8]{kollar1999} applied to the quadric bundle $\phi^{(r)}\colon \mathcal{Q}^{(r)}\rightarrow \P^1$. 
Therefore $F_r(X)$ is $\R$-unirational by Theorem~\ref{thm:symmetric}.
\end{proof}

\begin{rem}
A potential strategy for further extending Theorem~\ref{thm:fano-unirational} to $F_r(X)$ for arbitrary $r>1$ over arbitrary fields would be to reduce to showing separable $k$-unirationality of the surface $\mathcal{Q}^{(r)}$ for $X\subset \P^{2r+4}$. 
Nevertheless, for $r>1$, it is not even clear whether $\mathcal{Q}^{(r)}$ has a $k$-point.
The difficulty lies in the fact that a conic bundle over $\P^1$ with a $0$-cycle of degree $1$ does not necessarily have a $k$-point,
as first observed by Colliot-Th\'el\`ene--Coray \cite{CTC79}.
\end{rem}

\subsection{Rationality results for \(F_r(X)\) over certain fields}\label{section-specific-fields}
In this section, we give some consequences of Theorems~\ref{Fanoscheme} and~\ref{thm:symmetric} over some specific fields \(k\). 
First, we prove the following result, which generalizes the $r=0$ case due to  Colliot-Th\'el\`ene--Sansuc--Swinnerton-Dyer \cite[Theorem 3.4]{CTSSD}.

\begin{cor}\label{cor:fano-scheme-rational}
Over a field \(k\) of characteristic \(\neq 2\),
let $X$ be a smooth complete intersection of two quadrics in $\P^N$.
\begin{enumerate}
\item\label{item:Ci} If $k$ is a $C_i$-field for some $i\geq 0$, then, for every $0\leq r\leq \lfloor\frac{N}{2}\rfloor -2^i-1$, $F_r(X)$ is $k$-rational.
In particular, if \(k\) is algebraically closed, then \(F_r(X)\) is rational for all $0\leq r\leq \lfloor\frac{N}{2}\rfloor -2$.
\item\label{item:p-adic} If $k$ is a $p$-adic field, then, for every $0\leq r\leq \lfloor \frac{N}{2}\rfloor -5$, $F_r(X)$ is $k$-rational.
\item\label{item:totally-imaginary} If $k$ is a totally imaginary number field, then, for every $0\leq r\leq \lfloor\frac{N}{3}\rfloor -5$, $F_r(X)$ is $k$-rational.
\end{enumerate}
\end{cor}
\begin{proof}
First, we prove these statements under the assumption that \(F_r(X)(k)\neq\emptyset\).

\eqref{item:Ci}: The function field $k(\P^1)$ is a $C_{i+1}$-field by the Lang--Nagata theorem.
Since $\phi^{(r)}\colon\mathcal{Q}^{(r)}\rightarrow \P^1$ is a fibration into quadrics in $N-2r-1>2^{i+1}$ variables, the assumption implies it has a section.
Hence, $\mathcal{Q}^{(r)}$ is $k$-rational, so
$F_r(X)$ is $k$-rational by Theorem~\ref{thm:symmetric}.

\eqref{item:p-adic}: Let $k$ be a $p$-adic field.
By theorems of Parimala--Suresh \cite[Theorem 4.6]{parimalasuresh2010} (for $p\neq 2$) and Leep \cite[Theorem 3.4]{Leep2013} (for $p=2$; see also~\cite{ParimalaSuresh14})
on the $u$-invariant of quadratic forms over $p$-adic function fields,
any quadric bundle of relative dimension $\geq 7$  
over a curve defined over $k$ has a section.
Applying this to $\phi^{(r)}\colon \mathcal{Q}^{(r)}\rightarrow \P^1$, together with  
Theorem~\ref{thm:symmetric},
yields the result.

\eqref{item:totally-imaginary}:
Let $k$ be a totally imaginary number field.
If $-1\leq r\leq \lfloor\frac{N}{3}\rfloor -5$, the result of Leep \cite[Theorem 2.7]{Leep84} implies that $\left\{l_{ij}=q_i =0\, (i=0,1, j=0,\cdots,r)\right\}\subset \P^{N-r-1}$ has a $k$-point,
which yields a section for $\phi^{(r)}\colon \mathcal{Q}^{(r)}\rightarrow \P^1$, showing that $\mathcal{Q}^{(r)}$ is $k$-rational.
Theorem~\ref{thm:symmetric} then implies $F_r(X)$ is $k$-rational.

It remains to show by induction on \(r\) that \(F_r(X)(k)\neq\emptyset\) in each of the above cases. 
First assume $r=0$.
In each of the above cases, $\phi=\phi^{(-1)}$ has a section, hence $X(k)\neq \emptyset$ by \cite[Theorem 2.2]{Leep}.
Now assume \(r\geq 1\) and that \(F_{r-1}(X)(k)\neq\emptyset\). In each of the above cases, \(\phi^{(r-1)}\) has a section, so \(F_r(X)(k)\neq\emptyset\) by Proposition~\ref{HT}.
\end{proof}

If $k=\F_q$ is a finite field, 
the $i=1$ case of Corollary~\ref{cor:fano-scheme-rational}\eqref{item:Ci} applies,
so the only case when $\F_q$-rationality of $F_r(X)$ is undetermined is when $r=\lfloor \frac{N}{2}\rfloor-2$.
When $N$ is odd, we show $\F_q$-rationality for the remaining case.

\begin{cor}\label{thm:fano-scheme-rational-finite}
Let $\bb F_q$ be a finite field of characteristic \(\neq 2\).
Fix $g\geq 2$, and let $X$ be a smooth complete intersection of two quadrics in $\P^{2g+1}$.
Then, for every $0\leq r\leq g-2$, $F_r(X)$ is $\bb F_q$-rational.
\end{cor}
\begin{proof}
By Corollary~\ref{cor:fano-scheme-rational}\eqref{item:Ci}, it remains to consider the case $r=g-2$.
It follows from \cite[Theorem 4.8]{Reid-thesis} that $F_{g-1}(X)$ is a torsor under an abelian variety over \(\bb F_q\),
so it has a \(\bb F_q\)-point by Lang's theorem.
Now the result is immediate from Theorem~\ref{Fanoscheme}.
\end{proof}

When $N$ is even, the analogue of Corollary~\ref{thm:fano-scheme-rational-finite} does not hold: the second maximal Fano scheme \(F_{N/2 - 2}(X)\) is not necessarily rational over finite fields. 
Indeed, for $N=4$, over any finite field $\F_q$ there exist degree 4 del Pezzo surfaces that are not \(\bb F_q\)-rational \cite[Theorem 3.2]{rybakov05}.

Finally, over an algebraically closed field of characteristic $\neq 2$, let $C$ be a hyperelliptic curve of genus $g$.
One can then associate to $C$ a smooth complete intersection $X$ of two quadrics in $\P^{2g+1}$, uniquely determined up to isomorphism, such that the Stein factorization of $F_{g}(\mathcal{Q}/\P^1)\rightarrow \P^1$ yields $C$.
In \cite{Ramanan81}, Ramanan established an isomorphism between various moduli spaces of orthogonal and spin bundles on $C$ and the Fano schemes of linear spaces on $X$.
More precisely, let $-1$ be the hyperelliptic involution
on $C$ and fix an $(-1)$-invariant line bundle $\mathcal{L}$ of degree $2g-1$ on $C$.
For $1\leq n\leq g$, let $U_{n,\xi}$ be the moduli space of $(-1)$-invariant orthogonal bundles $\mathcal{E}$ of rank $2n$ with $\Gamma^+(2n)$-structure such that, for every Weierstrass point $x\in C$, the eigenspace associated to the eigenvalue $-1$ on $\mathcal{E}\otimes \mathcal{L}(x)$ has dimension $1$.
Here $\Gamma^+(2n)$ is a certain subgroup of the group of units of the Clifford algebra introduced in \cite[Section 2]{Ramanan81}.
The moduli space $U_{n,\xi}$ is isomorphic to $F_{g-n}(X)$ by \cite[Theorem 3]{Ramanan81}, so using Corollary~\ref{cor:fano-scheme-rational}\eqref{item:Ci}, we obtain:
\begin{cor}\label{thm:modulirational}
Over an algebraically closed field of characteristic $\neq 2$,
let $C$ be a hyperelliptic curve of genus $g$. 
Then, for every $2\leq n\leq g$, the moduli space $U_{n,\xi}$ is rational.
\end{cor}

\subsection{Applications to rational points}\label{sec:arithmetic-applications}

The following applications were kindly suggested to us by Jean-Louis Colliot-Th\'el\`ene.

\begin{cor}\label{cor:brauer-surjective}
Over a field $k$ of characteristic $0$, let $X$ be a smooth complete intersection of two quadrics in $\P^N$.
Let $0\leq r<\frac{N}{2}-2$.
Then $\Br(k)\rightarrow \Br(F_r(X))$ is surjective.
\end{cor}
\begin{proof}
The Leray spectral sequence yields an exact sequence:
\[
\Br(k)\rightarrow  \Ker(\Br(F_r(X))\rightarrow \Br(F_r(X_{\kbar})))\rightarrow H^1(k,\Pic(F_r(X_\kbar))).
\]
Corollary \ref{cor:fano-scheme-rational}\eqref{item:Ci} shows that $F_r(X_{\kbar})$ is rational, hence $\Br(F_r(X_{\kbar}))=0$.
Moreover, we have $\Pic(F_r(X_{\kbar}))=\Z$ by \cite[Corollaire 3.5]{DebarreManivel98}, where the Galois group acts trivially, thus $H^1(k,\Pic(F_r(X_\kbar)))=0$.
The result is now immediate from the above exact sequence.
\end{proof}

\begin{rems}\label{CT-remark} \hfill
\begin{enumerate}
\item A conjecture of Colliot-Th\'el\`ene states that the Brauer--Manin obstruction is the only obstruction to the Hasse principle for smooth projective rationally connected varieties over a number field.
By Corollary \ref{cor:brauer-surjective}, the conjecture predicts the Hasse principle for rational points on $F_r(X)$ for smooth complete intersections of two quadrics $X\subset \P^N$ and $0\leq r<\frac{N}{2}-2$.
\item 
If the assumptions on \(r\) in Corollary~\ref{cor:brauer-surjective} do not hold, then in general the Hasse principle fails for \(F_r(X)\). For \(N=2g\), \(r=g-2\), and \(g=2\), degree 4 del Pezzo surfaces do not satisfy the Hasse principle \cite{BSD-dP4}. For \(N=2g+1\) and \(r=g-1\), this is related to the Tate--Shafarevich group of the Jacobian of \(C\), see \cite[Section 8]{BGW} (see also \cite{FreiJi} for an explicit example when \(g=2\)).
\item 
For \(0\leq r\leq \lfloor \frac{N}{2}\rfloor -4\), the proof of Corollary~\ref{cor:fano-scheme-rational}\eqref{item:p-adic} shows that \(F_r(X)(\bb Q_p)\neq\emptyset\). However, \(F_r(X)\) may not have \(\bb R\)-points in general; see Lemma~\ref{lem:width-frequency}\eqref{item:fano-width}.
\end{enumerate}
\end{rems}

The following result generalizes the \(r=0\) case in \cite[Theorem A(ii)]{CTSSD}:

\begin{cor}\label{cor:weak-approximation-fano}
    Over a number field $k$, let $X$ be a smooth complete intersection of two quadrics in $\P^N$. Let $0\leq r<\frac{N}{2}-2$. If \(F_r(X)(k)\neq\emptyset\), then weak approximation holds for \(F_r(X)\).
\end{cor}

\begin{proof}
    The assumptions imply that \(\cal Q^{(r)}\) is defined and that the quadric fibration \(\phi^{(r)}\colon \cal Q^{(r)}\to\bb P^1\) has relative dimension \(\geq 2\); in particular, \(\cal Q^{(r)}\) is smooth and geometrically integral (Lemma~\ref{lem:hyperbolic-reduction}).
    Then \cite[Proposition 3.9]{CTSSD} (in the relative dimension \(\geq 3\) case) and \cite[Theorem 3.10]{CTSSD} (in the relative dimension \(2\) case; note that the assumptions of the theorem are satisfied by Lemma~\ref{lem:hyperbolic-reduction}\eqref{item:hyperbolic-reduction-pencil}
    and Lemma~\ref{lem:quadric-bundle-simple-degeneration}) 
    show that \(\cal Q^{(r)}\) satisfies weak approximation (and the Hasse principle).
    By Theorem~\ref{thm:symmetric}, \cite[Theorem 1.3]{ChenZhang23}, and Corollary~\ref{cor:brauer-surjective}, this implies that \(F_r(X)\) satisfies weak approximation.
\end{proof}

\section{
The odd-dimensional case
}\label{section-maximal}
We now focus on the odd-dimensional case and prove Theorem~\ref{maximallinearspace}.
Throughout this section,
fix $g\geq 2$ and let $X$ be a smooth complete intersection of two quadrics in $\P^{2g+1}$.
By Lemma~\ref{lem:Kuznetsov_fanoscheme}, the maximal linear subspaces on \(X\) are \((g-1)\)-planes. The Stein factorization of the relative Fano scheme $F_{g}(\mathcal{Q}/\P^1)\rightarrow \P^1$ yields a hyperelliptic curve $C$ of genus $g$.
Before starting the proof of Theorem~\ref{maximallinearspace}, we outline the main ideas.

For arbitrary \(g \geq 2\), the hyperbolic reduction \(\cal Q^{(g-2)}\) is a threefold. We show its intermediate Jacobian (Section~\ref{sec:CH2-scheme}) is the Jacobian of \(C\). A key result we need to prove is to identify \(F_{g-1}(X)\) with an intermediate Jacobian torsor given by a certain algebraic curve class on \(\cal Q^{(g-2)}\) (Proposition~\ref{prop-fanoch2s}), which we do by studying actions of correspondences coming from hyperbolic reductions. Using the explicit description of the curve classes on \(\cal Q^{(g-2)}\), we prove directly that \(2[F_{g-1}(X)]=[\bPic^1_{C/k}]\) (Proposition~\ref{prop-pic1ch2f} is an important step toward this).
To prove Theorem~\ref{maximallinearspace},
if the threefold \(\cal Q^{(g-2)}\) is \(k\)-rational, then the vanishing of the intermediate Jacobian torsor obstruction over \(k\) (Theorem~\ref{thm-BWIJ}) implies that the torsor \(F_{g-1}(X)\) is \(\bPic^d_{C/k}\) for some \(d\). Combining this with the property that \(2[F_{g-1}(X)]=[\bPic^1_{C/k}]\) shows \(F_{g-1}(X)(k)\neq\emptyset\).

When \(g=2\), \(\cal Q^{(0)}\) is \(k\)-birational to the threefold \(X \subset \bb P^5\). In this case, Hassett--Tschinkel and Benoist--Wittenberg showed that \(k\)-rationality of \(X\) implies \(F_1(X)(k)\neq\emptyset\) \cite{HT-intersection-quadrics,BW-IJ}. In their case, the universal line on \(X\) can be used to show that \(F_1(X)\) is an intermediate Jacobian torsor, and a result of Wang shows that \(2[F_1(X)]=[\bPic^1_{C/k}]\) as torsors \cite{Wang-maximal-linear-spaces}. In the present paper, we do not rely on Wang's result to prove Theorem~\ref{maximallinearspace}, as we give a direct argument that \(2[F_{g-1}(X)]=[\bPic^1_{C/k}]\). In fact, using Wang's result does not significantly simplify the proof of Theorem~\ref{maximallinearspace}, because we still need to prove that \(F_{g-1}(X)\) is an intermediate Jacobian torsor for arbitrary \(g\) (Proposition~\ref{prop-fanoch2s}).

\subsection{Preparation for the proof of Theorem~\ref{maximallinearspace}}
First, we consider the Fano scheme \(F_{g-1}(X)\) of maximal linear spaces.
By \cite[Theorem 1.1]{Wang-maximal-linear-spaces},
$F_{g-1}(X)$ is a torsor under $\bPic^0_{C/k}$.
The following weaker statement, which follows from a theorem of Reid \cite[Theorem 4.8]{Reid-thesis}, is enough for our application.

\begin{lem}[Corollary of {\cite[Theorem 4.8]{Reid-thesis}}]\label{lem:fanotorsor}
$F_{g-1}(X)$ is a torsor under $\bAlb_{F_{g-1}(X)/k}$.
In particular, $\dim \bAlb_{F_{g-1}(X)/k}=g$.
\end{lem}
\begin{proof}
We will show that the natural map $F_{g-1}(X)\rightarrow \bAlb^1_{F_{g-1}(X)/k}$ is an isomorphism.
(See \cite[Section 2]{wittenberg08} for the definition and basic properties of the Albanese torsor $\bAlb^1_{F_{g-1}(X)/k}$.)
By the universality of the map $F_{g-1}(X)\rightarrow \bAlb^1_{F_{g-1}(X)/k}$, it is enough for us to show that $F_{g-1}(X_{\kbar})$ is isomorphic to an abelian variety over $\kbar$.
This follows from the fact that $F_{g-1}(X_{\kbar})$ is isomorphic to $\bPic_{C_{\kbar}/\kbar}^0$ by \cite[Theorem 4.8]{Reid-thesis}.
\end{proof}

For the remainder of this section, fix \(0\leq r \leq g-1\) and assume $F_r(X)(k)\neq \emptyset$. Choose $\ell\in F_r(X)(k)$, and let $\mathcal{Q}^{(r)}$ and \(E^{(r)}\) be as defined in Section~\ref{sec:Q^(r)-construction}.
We study the action $(E^{(r)})_*$ on algebraic cycles.
For this purpose, it is useful to 
regard the hyperbolic reduction with respect to $\ell$ as iterations of the hyperbolic reduction with respect to a point, as follows.

Choose $k$-points $p_0,\cdots, p_r \in \ell$ with $\langle p_0,\cdots, p_r\rangle =\ell$
and coordinate points for $\P^{2g+1}$ that extend $p_0,\cdots, p_r$,
so that we get coordinate expressions of the key varieties as in Section~\ref{section-fano-hyperbolic}.
Let $l_{0j}, q_0, l_{1j}, q_1\in k[x_{r+1}, \dots, x_{2g+1}]$ be the forms defined in Section~\ref{sec:Q^(r)-construction}.
For each $0 \leq i \leq r$, let
\[
\overline{\mathcal{P}^{(i-1)}}\coloneqq
\left\{
\begin{pmatrix}
s & t
\end{pmatrix}
\begin{pmatrix}
l_{00} & \dots &l_{0 , i-1}\\
l_{10} & \dots & l_{1 , i-1}
\end{pmatrix}
=0
\right\}
\subset \P^1\times \P^{2g-i} ,
\]
where the
coordinates on \(\mathbb P^{2g-i}\) are \([y_{i+1}:\cdots:y_{2g+1}]\).
Note that \(\overline{\mathcal{P}^{(i-1)}}\) is the image of \(\mathcal P^{(i-1)}\subset\mathbb P^1\times\mathbb P^{2g-i+1}\) under \(\id\times \pi_{p_i}\), which informs this choice of notation. Then, the hyperbolic reduction $\mathcal{Q}^{(i)}$ of $\phi\colon\mathcal{Q}\rightarrow \P^1$ with respect to $\ell_i \coloneqq \langle p_0,\cdots, p_i\rangle$ is given by
\[\mathcal{Q}^{(i)}
=
\left\{
\begin{pmatrix}
s & t
\end{pmatrix}
\begin{pmatrix}
l_{00} & \dots &l_{0i} & q_0 + \sum_{j=i+1}^r l_{0j} y_j \\
l_{10} & \dots & l_{1i} & q_1 + \sum_{j=i+1}^r l_{1j} y_j
\end{pmatrix}
=0
\right\}
\subset \P^1\times \P^{2g-i}.\]
We have a blow-up diagram (see \cite[Remark 2.6]{KuznetsovShinder18})
\[
\begin{tikzcd}
& \widetilde{\mathcal{Q}^{(i-1)}} \arrow[ld, "\bl_{\P^1\times p_{i}}"'] \arrow[rd, "\bl_{\mathcal{Q}^{(i)}}"] & \arrow[l, hook'] E_{i} \arrow[rd, "\pi_i"] &\\
\mathcal{Q}^{(i-1)} \arrow[rr, "\id\times \pi_{p_{i}}", dashed]& & \overline{\mathcal{P}^{(i-1)}} & \arrow[l, hook']\mathcal{Q}^{(i)}    , 
\end{tikzcd}
\]
where $E_{i}$ is the exceptional divisor of $\bl_{\mathcal{Q}^{(i)}}$
and $\pi_i\colon E_i\rightarrow \mathcal{Q}^{(i)}$ is the projection.

\begin{lem}\label{lem:chow-onestep-reduction}
The following statements hold.
\begin{enumerate}
\item\label{item:chow-onestep-reduction-exc-corresp} $E^{(r)}=E_r\circ\dots \circ E_0$ as correspondences, where we regard $E^{(r)}$ (resp. $E_i$) as a correspondence on $\mathcal{Q}\times \mathcal{Q}^{(r)}$ (resp. as a correspondence on $\mathcal{Q}^{(i-1)}\times \mathcal{Q}^{(i)}$).
\item\label{item:chow-onestep-reduction-CHQ} 
For every $0 \leq i \leq r$, there is an isomorphism
\[
(E_i)_*\colon \CH^{g-i+1}(\mathcal{Q}^{(i-1)})_{\alg}\xrightarrow{\sim}\CH^{g-i}(\mathcal{Q}^{(i)})_{\alg}, \quad \alpha \mapsto (\pi_i)_*((\bl_{\P^1\times p_i}^*\alpha)|_{E_i}).
\]
The inverse, which we denote by $-(E_i^T)_*$, is given by 
$\beta \mapsto -(\bl_{\P^1\times p_i})_*(\pi_i)^*\beta$.
\end{enumerate}
\end{lem}
\begin{proof}
\eqref{item:chow-onestep-reduction-exc-corresp}:
Using Lemma~\ref{lem:hyperbolic-reduction}\eqref{item:hyperbolic-reduction-embedding} to identify $\mathcal{Q}^{(i)}$ with the relative Fano scheme of isotropic $(i+1)$-planes of $\phi\colon\mathcal{Q}\rightarrow \P^1$ containing $\ell_i=\langle p_0,\dots, p_i\rangle$, we may regard $E^{(i)}$ as a subscheme of $\mathcal{Q}\times \mathcal{Q}^{(i)}$ as follows:
\[E^{(i)}=\left\{(x,L)\mid x \in L\right\}\subset \mathcal{Q}\times_{\P^1} \mathcal{Q}^{(i)}\subset \mathcal{Q}\times \mathcal{Q}^{(i)}.\]
Similarly, identifying $\mathcal{Q}^{(i)}$ with the relative Fano scheme of isotropic lines of $\phi^{(i-1)}\colon\mathcal{Q}^{(i-1)}\rightarrow \P^1$ containing $p_i$,
we may regard $E_i$ as a subscheme of $\mathcal{Q}^{(i-1)}\times \mathcal{Q}^{(i)}$ as follows:
\[
E_i=\left\{(y, m)\mid y\in m \right\}\subset \mathcal{Q}^{(i-1)}\times_{\P^1}
\mathcal{Q}^{(i)}\subset \mathcal{Q}^{(i-1)}\times \mathcal{Q}^{(i)}.\]
A line on a fiber of $\phi^{(i-1)}\colon\mathcal{Q}^{(i-1)}\rightarrow \P^1$ and containing $p_i$ corresponds to an $(i+1)$-plane on a fiber of $\phi\colon\mathcal{Q}\rightarrow \P^1$ and containing $\ell_i$, so $E^{(i)}=E_i\circ E^{(i-1)}$.
Since \(E^{(0)} = E_0\), we have $E^{(r)}=E_r\circ E^{(r-1)}=\dots =E_r\circ\dots\circ E_0$.

\eqref{item:chow-onestep-reduction-CHQ}:
For $0\leq i\leq g-2$,
the blow-up formula for Chow groups implies
\begin{align*}
\CH^{g-i+1}(\widetilde{\mathcal{Q}^{(i-1)}})
&=\CH^{g-i+1}(\mathcal{Q}^{(i-1)})\oplus \CH^0(\P^1)\oplus \CH^1(\P^1)\\
&=\CH^{g-i+1}(\overline{\mathcal{P}^{(i-1)}})\oplus \CH^{g-i}(\mathcal{Q}^{(i)}),
\end{align*}
where $(\bl_{\P^1\times p_i})_*?$ and $-\pi_*(?|_{E_i})$ respectively define the projectors onto $\CH^{g-i+1}(\mathcal{Q}^{(i-1)})$ and $\CH^{g-i}(\mathcal{Q}^{(i)})$.
For $i=g-1$,
we have the same decomposition except that $\CH^0(\P^1)$ does not appear due to dimension reasons.
Finally, note that $\CH^{0}(\P^1)_{\alg} = \CH^1(\P^1)_{\alg} = \CH^{g-i+1}(\overline{\mathcal{P}^{(i-1)}})_{\alg}=0$ since \(\overline{\mathcal{P}^{(i-1)}}\) is a projective bundle over \(\bb P^1\).
\end{proof}

In particular, Lemma~\ref{lem:chow-onestep-reduction} implies:
\begin{cor}\label{cor:Er-action-on-alg}
There is an isomorphism
\[
(E^{(r)})_*\colon \CH^{g+1}(\mathcal{Q})_{\alg}\xrightarrow{\sim} \CH^{g-r}(\mathcal{Q}^{(r)})_{\alg}
\]
with inverse given by $(-1)^{r+1}(E^{(r)T})_*$.
\end{cor}

\subsection{Proof of Theorem~\ref{maximallinearspace}}
In the following, we assume $F_{g-2}(X)(k)\neq \emptyset$.
For any $\ell\in F_{g-2}(X)(k)$, Lemma~\ref{lem:hyperbolic-reduction}\eqref{item:hyperbolic-reduction-smooth} and~\eqref{item:hyperbolic-reduction-independence}
show that $\mathcal{Q}^{(g-2)}$ is a smooth projective geometrically connected threefold,
whose $k$-birational equivalence class does not depend on $\ell$.
We start by showing that the $k$-isomorphism class of the intermediate Jacobian
$(\bCH^2_{\mathcal{Q}^{(g-2)}/k})^0$ (see Section~\ref{sec:CH2-scheme}) does not depend on $\ell$.

\begin{lem}\label{lem:ch2-ll'-isom}
For $\ell, \ell'\in F_{g-2}(X)(k)$, let $\mathcal{Q}^{(g-2)}_\ell, \mathcal{Q}^{(g-2)}_{\ell'}$ denote the corresponding hyperbolic reductions of $\mathcal{Q}$,
and let $E^{(g-2)}_{\ell}, E^{(g-2)}_{\ell'}$ denote the exceptional subschemes.
Then the composition $(-1)^{g-1}E^{(g-2)}_{\ell'}\circ E^{(g-2)T}_\ell$
induces an isomorphism of
abelian varieties
\[
(\bCH^2_{\mathcal{Q}^{(g-2)}_\ell/k})^0\xrightarrow{\sim}(\bCH^2_{\mathcal{Q}^{(g-2)}_{\ell'}/k})^0.
\]
\end{lem}
\begin{proof}
Setting $\Gamma \coloneqq (-1)^{g-1}E^{(g-2)}_{\ell'}\circ E^{(g-2)T}_{\ell}$,
we will show that the morphisms 
\begin{align*}
\Gamma_*\colon (\bCH^2_{\mathcal{Q}^{(g-2)}_\ell/k})^0\rightarrow (\bCH^2_{\mathcal{Q}^{(g-2)}_{\ell'}/k})^0, \quad (\Gamma^T)_*\colon 
(\bCH^2_{\mathcal{Q}^{(g-2)}_{\ell'}/k})^0\rightarrow (\bCH^2_{\mathcal{Q}^{(g-2)}_\ell/k})^0
\end{align*}
are inverse to each other.
By Corollary~\ref{cor:Er-action-on-alg},
the compositions $(\Gamma^T)_*\Gamma_*$ and $\Gamma_*(\Gamma^T)_*$ are the identities on the groups
$\CH^2(\mathcal{Q}^{(g-1)}_\ell)_{\alg}$ and $\CH^2(\mathcal{Q}^{(g-1)}_{\ell'})_{\alg}$, respectively.
Since the intermediate Jacobian
of a geometrically rational threefold geometrically agrees with Murre's intermediate Jacobian \cite[Theorem 3.1(vi)]{BW-IJ}, 
the universal property \cite[Section 1.8]{Murre85} implies that $(\Gamma^T)_*\Gamma_*$ and $\Gamma_*(\Gamma^T)_*$ are automorphisms of the abelian varieties $(\bCH^2_{\mathcal{Q}^{(g-1)}_{\ell}/k})^0$ and $(\bCH^2_{\mathcal{Q}^{(g-2)}_{\ell'}/k})^0$.
This concludes the proof.
\end{proof}

\begin{lem}\label{lem:Zbetti3}
$\dim (\bCH^2_{\mathcal{Q}^{(g-2)}/k})^0=g$.
\end{lem}
\begin{proof}
By passing to a finite extension of $k$,
we may assume $F_{g-1}(X)(k)\neq \emptyset$.
Using Lemma~\ref{lem:ch2-ll'-isom},
we may further assume that $\ell$ is contained in some $L\in F_{g-1}(X)(k)$.
After choosing a $k$-point in $L\setminus \ell$
and identifying $\mathcal{Q}^{(g-1)}$ with $C$ by Lemma~\ref{lem:hyperbolic-reduction}\eqref{item:hyperbolic-reduction-max-min}, the blow-up formula of Benoist--Wittenberg \cite[Proposition 3.10]{BW-IJ} yields an isomorphism of (principally polarized) abelian varieties
\[
(E_{g-1})_*\colon (\bCH^2_{\mathcal{Q}^{(g-2)}/k})^0\xrightarrow{\sim}\bPic_{C/k}^0.
\]
We conclude $\dim (\bCH^2_{\mathcal{Q}^{(g-2)}/k})^0=\dim \bPic_{C/k}^0=g$.
\end{proof}

\begin{lem}\label{lem:NS2-Q-g-2}
The inclusion $\iota\colon \mathcal{Q}^{(g-2)}\rightarrow \P^1\times \P^{g+2}$ induces an isomorphism \[\iota_*\colon \NS^2(\mathcal{Q}^{(g-2)}_{\kbar})\xrightarrow{\sim}\NS^{g+2}(\P^1_{\kbar}\times_{\kbar}\P^{g+2}_{\kbar}).\]
In particular, if \(s,f \in \NS^2(\mathcal{Q}^{(g-2)}_{\kbar})\) denote the classes defined by $s\mapsto \P^1\times*$ and $f\mapsto *\times \P^1$, we have \(\NS^2(\mathcal{Q}^{(g-2)}_{\kbar})=\Z s\oplus \Z f\).
\end{lem}
\begin{proof}
Over \(\kbar\), let \(s'\) denote the class of a section of the quadric surface fibration $\phi^{(g-2)}_\kbar \colon \mathcal{Q}^{(g-2)}_{\kbar}\rightarrow \P^1_{\kbar}$ (which exists by Tsen's theorem), and let \(f\) denote the class of a \(\kbar\)-line $f$ on a fiber of \(\phi^{(g-2)}_\kbar\).
We show that
\begin{align}\label{eq-map-to-NS}
\Z^2\rightarrow \NS^2(\mathcal{Q}^{(g-2)}_{\kbar}), \quad (a,b)\mapsto a s' + bf
\end{align}
is an isomorphism.
Since $\mathcal{Q}^{(g-2)}_{\kbar(\P^1)}$ is an isotropic quadric, $\CH^2(\mathcal{Q}^{(g-2)}_{\kbar(\P^1)})=\Z$ is generated by the image of $s'$.
Moreover, the Chow group of $1$-cycles on any fiber of $\phi^{(g-2)}_\kbar\colon\mathcal{Q}^{(g-2)}_{\kbar}\rightarrow \P^1_{\kbar}$ is generated by the classes of \(\kbar\)-lines on it.
The localization exact sequence then yields that every class $\alpha \in \CH^2(\mathcal{Q}^{(g-2)}_{\kbar})$ may be written as
\[
\alpha = \deg(\alpha/\P^1)s' + b_1f_1+\dots + b_nf_n,
\]
where $b_1,\dots, b_n\in \Z$ and $f_1, \dots, f_n$ are \(\kbar\)-lines on fibers of \(\phi^{(g-2)}_\kbar\).
Since $F_1(\mathcal{Q}^{(g-2)}_{\kbar}/\P^1_{\kbar})$ is a $\P^1$-bundle over $C$, 
we have $f_1=\dots =f_n=f$ in $\NS^2(\mathcal{Q}^{(g-2)}_{\kbar})$ and the map \eqref{eq-map-to-NS} is surjective. 
The injectivity of \eqref{eq-map-to-NS} follows from the fact that the images of $s',f$ in $\NS^{g+2}(\P^1_{\kbar}\times_{\kbar}\P^{g+2}_{\kbar})$ are linearly independent.
Finally, this implies \(\iota_*\) is an isomorphism because \(\iota_* s'\) and \(\iota_* f\) freely generate \(\NS^{g+2}(\P^1_{\kbar}\times_{\kbar}\P^{g+2}_{\kbar})\).
\end{proof}

Now we can describe the intermediate Jacobian \((\bCH^2_{\mathcal{Q}^{(g-2)}/k})^0\) and the torsor associated to the class \(f\).

\begin{prop}\label{prop-pic1ch2f}
There exists an isomorphism of principally polarized abelian varieties
\[
\bPic_{C/k}^0\xrightarrow{\sim}(\bCH^2_{\mathcal{Q}^{(g-2)}/k})^0,
\]
which underlies an isomorphism of torsors
\[
\bPic_{C/k}^1\xrightarrow{\sim}(\bCH^2_{\mathcal{Q}^{(g-2)}/k})^f.
\]

\end{prop}
\begin{proof}
Denote $F \coloneqq F_1(\mathcal{Q}^{(g-2)}/\P^1)$, and let $U \coloneqq U_1(\mathcal{Q}^{(g-2)}/\P^1)$ be the universal family. Let $p\colon F\rightarrow C$ be the morphism induced by the Stein factorization of the natural morphism $F\rightarrow \P^1$,
and let $-1\colon C\rightarrow C$ be the hyperelliptic involution.
We have morphisms of torsors
\[p_*\colon \bAlb_{F/k}^1\rightarrow \bPic_{C/k}^1, \quad
U_*\colon \bAlb_{F/k}^1\rightarrow (\bCH^2_{\mathcal{Q}^{(g-2)}/k})^f
\] 
which induce morphisms of abelian varieties
\[p_*\colon \bAlb_{F/k}\rightarrow \bPic_{C/k}^0, \quad
U_*\colon \bAlb_{F/k}\rightarrow (\bCH^2_{\mathcal{Q}^{(g-2)}/k})^0.
\]
Since $p\colon F\rightarrow C$ is geometrically a $\P^1$-bundle, $p_*\colon \bAlb_{F/k}\rightarrow \bPic_{C/k}^0$ is an isomorphism of abelian varieties.
It now remains for us to show that the composition 
$U_*\circ (p_*)^{-1}\colon \bPic_{C/k}^0\rightarrow (\bCH^2_{\mathcal{Q}^{(g-2)}/k})^0$
is an isomorphism of principally polarized abelian varieties.
For this, we may assume that $k$ is an algebraically closed.
We have the following diagram.
\[
\begin{tikzcd}
\bPic^0_{C/k}                  & \bAlb_{F/k} \arrow[l, "p_*"'] \arrow[rd, "U_*"] \arrow[dd, "U^*\circ U_*"] &                                                \\
                               &                                                                      & (\bCH^2_{\mathcal{Q}^{(g-2)}/k})^0 \arrow[ld, "U^*"] \\
\bPic^0_{C/k} \arrow[r, "p^*"] & \bPic_{F/k}^0                                                        &                                               
\end{tikzcd}
\]
The composition $(p^*)^{-1}\circ U^*\circ U_*\circ (p_*)^{-1}$ is equal to $(-1)^*$,
and thus is an isomorphism.
Since Lemma~\ref{lem:Zbetti3} shows 
$\dim \bPic_{C/k}^0=\dim (\bCH^2_{\mathcal{Q}^{(g-2)}/k})^0=g$,
all the arrows in the above diagram are isomorphisms of abelian varieties.
To show that $U_*\circ (p_*)^{-1}$ respects the principal polarizations on $\bPic_{C/k}^0$ and $(\bCH^2_{\mathcal{Q}^{(g-2)}/k})^0$,
consider the following diagram of $\ell$-adic cohomology where all arrows are isomorphisms.
\[
\begin{tikzcd}
H^1(C) \arrow[dd, "(-1)^*"']                 & H^3(F) \arrow[l, "p_*"'] \arrow[rd, "U_*"] \arrow[dd, "U^*\circ U_*"] &                                                \\
                               &                                                                      & H^3(\mathcal{Q}^{(g-2)}) \arrow[ld, "U^*"] \\
H^1(C) \arrow[r, "p^*"] & H^1(F)                                                        &                                               
\end{tikzcd}
\]
For $\alpha, \beta\in H^3(F)$,
\[
(U_*\alpha) \cup (U_*\beta) = \alpha\cup (U^*U_*\beta) = (p_*\alpha)\cup ((-1)^* p_*\beta) = - (p_*\alpha) \cup (p_*\beta),
\]
where we have used that the pullback $(-1)^*$ on $H^1(C)$ is equal to multiplication by $-1$.
The characterization of the principal polarization on the intermediate Jacobian \((\bCH^2)^0\) in \cite[Property 2.4, the following comments, and Identity (2.9)]{BW20} then concludes the proof.
\end{proof}

Next, we will use the following lemmas to relate \(F_{g-1}(X)\) to the intermediate Jacobian \((\bCH^2_{\mathcal{Q}^{(g-2)}/k})^0\).

\begin{lem}\label{lem:L-to-section}
Let $L \in \NS^g(X_{\kbar})$ be the class of a $(g-1)$-plane on $X$.
Then there exists $a\in \Z$ such that
\[
(E^{(g-2)})_*(\P^1\times L) = s+af
\]
in $\NS^2(\mathcal{Q}^{(g-2)}_{\kbar})$, where \(s\) and \(f\) are the classes defined in Lemma~\ref{lem:NS2-Q-g-2}.
\end{lem}
\begin{proof}
On the generic fiber level, each $(E_i)_*$ corresponds to taking the intersection with the tangent hyperplane $H_i$ to $\mathcal{Q}^{(i-1)}_{k(\P^1)}$ at $p_i$ and then projecting to the base of the cone $\mathcal{Q}^{(i-1)}_{k(\P^1)}\cap H_i$.
Hence a linear space on $\mathcal{Q}^{(i-1)}_{k(\P^1)}$ maps to a linear space of one dimension lower on $\mathcal{Q}^{(i)}_{k(\P^1)}$.
In particular, under $(E^{(g-2)})_*$, a $(g-1)$-plane on $\mathcal{Q}_{k(\P^1)}$ maps to a $k(\P^1)$-point on $\mathcal{Q}_{k(\P^1)}$.
(Note: since we are interested in algebraic equivalence classes over $\kbar$, we may assume that a linear space on a quadric in question does not contain a blown-up point.)
\end{proof}

We will also need the following variant of Lemma~\ref{lem:L-to-section}.

\begin{lem}\label{lem:L-to-point}
Assume $F_{g-1}(X)(k)\neq\emptyset$. Choose $L\in F_{g-1}(X)(k)$ so that $\mathcal{Q}^{(g-1)}$ is defined, and, using Lemma~\ref{lem:hyperbolic-reduction}\eqref{item:hyperbolic-reduction-embedding},
identify $\mathcal{Q}^{(g-1)}$ with 
the relative Fano scheme of isotropic $g$-planes of $\mathcal{Q}\rightarrow \P^1$ containing $L$.
There exists a divisor class \(D\) on $\P^1\times \P^{g+1}$ such that for any $(g-1)$-plane $M$ on $X$ with $\dim L\cap M=g-2$,
\[
(E^{(g-1)})_*(\P^1\times M) = (-1)^{g-1}\langle L, M\rangle + \mathcal{Q}^{(g-1)}\cdot D
\]
in $\CH^1(\mathcal{Q}^{(g-1)})$.
\end{lem}
\begin{proof}
Fix a $(g-1)$-plane $M$ on $X$ such that $\dim L\cap M=g-2$.
Choose $k$-points $p_0,\dots, p_{g-1}\in L$ such that $L\cap M=\langle p_0,\dots, p_{g-2}\rangle$ and $L=\langle p_0,\dots, p_{g-1}\rangle$.
We inductively show that
for every $i\in\left\{0,\dots, g-2\right\}$
\[
(E^{(i)})_*(\P^1\times M)=(-1)^{i+1}V^{(i)}_M + \mathcal{Q}^{(i)}\cdot (\sum_{j=0}^i(-1)^{j+1}((j+1)H_1\cdot H_2^{g-i-1}-H_2^{g-i}))
\]
in $\CH^{g-i}(\mathcal{Q}^{(i)})$,
where $V^{(i)}_M$ is the subvariety of $\mathcal{Q}^{(i)}$ corresponding to $\P^1\times M$, and $H_1, H_2$ are the pull-backs of $\O(1)$ on $\P^1, \P^{2g-r}$ respectively.
This is a consequence of Lemma~\ref{lem:chow-onestep-reduction} and the equalities
\begin{align*}
&(E_i)_*V_M^{(i-1)}= - V_M^{(i)} + \mathcal{Q}^{(i)}\cdot(-(i+1)H_1^{g-i-1}\cdot H_2 + H_2^{g-i}),\\
&(E_i)_*(\mathcal{Q}^{(i-1)}\cdot (H_1\cdot H_2^{g-i}))=\mathcal{Q}^{(i)}\cdot (H_1\cdot H_2^{g-i-1}),\\
&(E_i)_*(\mathcal{Q}^{(i-1)}\cdot H_2^{g-i+1})=\mathcal{Q}^{(i)}\cdot H_2^{g-i},
\end{align*}
which may be directly verified.
Note that the last two formulas also hold for $i=g-1$.
Finally, using that there is a unique point $p\in \P^1$ such that $\langle L, M\rangle \subset \mathcal{Q}_p$,
\[
\langle L, M\rangle = (E_{g-1})_*V_M^{(g-2)},
\]
and we get the desired formula with
\[
D=\sum_{j=0}^{g-2}(-1)^{j+1}((j+1)H_1-H_2).
\]
\end{proof}

We are now ready to identify \(F_{g-1}(X)\) with an intermediate Jacobian torsor of \(\cal Q^{(g-2)}\).
\begin{prop}\label{prop-fanoch2s}
For the integer \(a\) in Lemma~\ref{lem:L-to-section},
there exists an isomorphism of 
$k$-schemes
\[F_{g-1}(X)\xrightarrow{\sim}(\bCH^2_{\mathcal{Q}^{(g-2)}/k})^{s+af}.\]
\end{prop}
\begin{proof}
Let \(U_{g-1}(X)\subset F_{g-1}(X)\times X\) denote the universal family associated to \(F_{g-1}(X)\).
In addition, embed $\P^1\times X$ into $X\times \mathcal{Q}$ by the projection $\P^1\times X\rightarrow X$ and the inclusion $\P^1\times X\subset\mathcal{Q}\subset \P^1\times \P^N$.
Then, if \(L\) is a \((g-1)\)-plane on \(X\), the image of its class under \(((\mathbb P^1\times X)\circ U_{g-1}(X))_*\colon \CH_0(F_{g-1}(X)) \to \CH^{g+1}(\mathcal Q)\) is the class of \(\mathbb P^1\times L\).
Recall that we view \(E^{(g-2)}\) as a correspondence on \(\mathcal Q\times\mathcal Q^{(g-2)}\). Then,
by Lemma~\ref{lem:L-to-section},
$E^{(g-2)}\circ (\P^1\times X)\circ U_{g-1}(X)$ induces a morphism of $k$-schemes
\[
(E^{(g-2)}\circ (\P^1\times X)\circ U_{g-1}(X))_*\colon
F_{g-1}(X)\rightarrow (\bCH^2_{\mathcal{Q}^{(g-2)}/k})^{s+af}.
\]
We aim to show that this is an isomorphism.
By 
Lemma~\ref{lem:fanotorsor},
it is enough for us to show that
\[
(E^{(g-2)}\circ (\P^1\times X)\circ U_{g-1}(X))_*\colon \bAlb_{F_{g-1}(X)/k}\rightarrow (\bCH^2_{\mathcal{Q}^{(g-2)}/k})^0
\]
is an isomorphism of abelian varieties.
To verify this, we may assume \(k\) is algebraically closed.
As in the proof of Lemma~\ref{lem:Zbetti3},
we may assume that $\ell$ is contained in some $(g-1)$-plane $L$ on $X$,
and after choosing a point in $L\setminus \ell$ and identifying $\mathcal{Q}^{(g-1)}$ with $C$ by Lemma~\ref{lem:hyperbolic-reduction}\eqref{item:hyperbolic-reduction-max-min},
we get an isomorphism of principally polarized abelian varieties
\[
(E_{g-1})_*\colon (\bCH^2_{\mathcal{Q}^{(g-2)}/k})^0\xrightarrow{\sim}\bPic_{C/k}^0.
\]
Moreover, 
Lemma~\ref{lem:hyperbolic-reduction}\eqref{item:hyperbolic-reduction-image}
yields a morphism
$C\rightarrow F_{g-1}(X)$, which induces
\[
\bPic_{C/k}^0\rightarrow \bAlb_{F_{g-1}(X)/k}.
\]
We claim that the composition
\[
\bPic^0_{C/k}\rightarrow \bAlb_{F_{g-1}(X)/k}\xrightarrow{(E^{(g-2)}\circ (\P^1\times X)\circ U_{g-1}(X))_*}(\bCH^2_{\mathcal{Q}^{(g-2)}/k})^0\xrightarrow{(E_{g-1})_*}\bPic^0_{C/k},
\]
which by Lemma~\ref{lem:chow-onestep-reduction} equals the composition
\begin{align}\label{eq-composition}
\bPic^0_{C/k}\rightarrow \bAlb_{F_{g-1}(X)/k}\xrightarrow{(E^{(g-1)}\circ (\P^1\times X)\circ U_{g-1}(X))_*}\bPic^0_{C/k},
\end{align}
is $(-1)^{g-1}$ times the identity, hence an isomorphism. Since $\dim \bPic_{C/k}^0=\dim \bAlb_{F_{g-1}(X)/k}=g$ 
by Lemma~\ref{lem:fanotorsor},
this will conclude the proof.

We may identify a given point of $C$ with a $(g-1)$-plane $M$ on $X$ such that $\dim L\cap M=g-2$, and also with a $g$-plane on a fiber of $\mathcal{Q}\rightarrow \P^1$ and containing $L$, where they correspond by 
$M\mapsto \langle L, M\rangle$.
The map in \eqref{eq-composition} is well-defined on the Picard group of $C$. By Lemma~\ref{lem:L-to-point},
the map may be described as
\[
M \mapsto (-1)^{g-1} \langle L, M\rangle + C\cdot D,
\]
where $D$ is a constant divisor class on $\P^1\times \P^{g+1}$
(recall that we identify \(C\) with \(\mathcal Q^{(g-1)} \subset\mathbb P^1\times\mathbb P^{g+1}\)).
The see-saw theorem then shows that the map in \eqref{eq-composition} is induced by a correspondence of the form
\[
(-1)^{g-1}\Delta_C + \beta\times C + C\times \gamma,
\]
where $\beta,\gamma \in \Pic(C)$.
(See also \cite[Corollary 4.12]{Reid-thesis}.)
Since $\beta\times C$ and $C\times \gamma$ act trivially on $\bPic^0_{C/k}$, this finishes the proof.
\end{proof}

The following shows that
the obstruction to the existence of a $(g-1)$-plane over $k$ is of order $4$.

\begin{lem}\label{lem:order4obstruction}
The following statements hold.
\begin{enumerate}
\item\label{item:order4obstruction-coker}
The cokernel of $\iota^*\colon \NS^2(\P^1_\kbar\times_\kbar \P^{g+2}_\kbar)\rightarrow \NS^2(\mathcal{Q}^{(g-2)}_{\kbar})$ is generated by \(s\) and is isomorphic to $\Z/4$. The equivalence $f\equiv 2s$ holds in this cokernel. 
\item\label{item:order4obstruction-torsors} As torsors,
$(\bCH^2_{\mathcal{Q}^{(g-2)}})^{f}$ and $(\bCH^2_{\mathcal{Q}^{(g-2)}})^{2s}$ are $k$-isomorphic, and $(\bCH^2_{\mathcal{Q}^{(g-2)}})^{2f} \cong (\bCH^2_{\mathcal{Q}^{(g-2)}})^{4s}$ is split over \(k\).
\end{enumerate}
\end{lem}
\begin{proof}
Using that \(\mathcal Q^{(g-2)} \subset\mathbb P^1\times\mathbb P^{g+2}\) is a complete intersection of $g-1$ divisors of type $(1,1)$ and a divisor of type $(1,2)$,
part~\eqref{item:order4obstruction-coker} follows by observing that
\[\P^1 \times \P^g \mapsto 2s+(2g-1)f, \quad
* \times \P^{g+1} \mapsto 2f\]
under 
$\NS^2(\P^1_\kbar\times_\kbar \P^{g+2}_\kbar)\rightarrow \NS^2(\mathcal{Q}^{(g-2)}_{\kbar})$.
Part~\eqref{item:order4obstruction-torsors} is immediate from part~\eqref{item:order4obstruction-coker} and properties of the \(\bCH^2\)-scheme (see Section~\ref{sec:CH2-scheme}).
\end{proof}

Finally, we are ready to prove Theorem~\ref{maximallinearspace}.
\begin{proof}[Proof of Theorem~\ref{maximallinearspace}]
The backward direction follows from Proposition~\ref{HT}, applied for $r=g-2$.
We show the forward direction.
If $\mathcal{Q}^{(g-2)}$ is $k$-rational, then
Theorem~\ref{thm-BWIJ} and Propositions~\ref{prop-pic1ch2f} and~\ref{prop-fanoch2s} imply that
there exists $d\in \Z$ such that
\begin{align}\label{eq-fano-n}
[F_{g-1}(X)]=[(\bCH^2_{\mathcal{Q}^{(g-2)}/k})^{s+af}]=[\bPic_{C/k}^d].
\end{align}
On the other hand, by Propositions~\ref{prop-pic1ch2f} and~\ref{prop-fanoch2s},
Lemma~\ref{lem:order4obstruction}\eqref{item:order4obstruction-torsors}, and additivity of \((\bCH^2)^0\)-torsors (see Section~\ref{sec:CH2-scheme}), we have
\[
2[F_{g-1}(X)]=[(\bCH^2_{\mathcal{Q}^{(g-2)}/k})^{2s+2af}]=[(\bCH^2_{\mathcal{Q}^{(g-2)}/k})^f]=[\bPic_{C/k}^1],
\]
which in turn implies
\begin{align}\label{eq-fano-1-n}
[F_{g-1}(X)]=[\bPic_{C/k}^{1-d}].
\end{align}
Since \(C\) is hyperelliptic, \([\bPic_{C/k}^{2m}]=0\) for any integer \(m\).
Since the parities of $d$ and $1-d$ are distinct,  
the equalities~\eqref{eq-fano-n} and~\eqref{eq-fano-1-n} imply 
$[F_{g-1}(X)]=0$, i.e., $F_{g-1}(X)$ has a \(k\)-point,
completing the proof.
\end{proof}

\begin{rem}
Let $\alpha_X \in \Br(C)$ be the Brauer class associated to the even Clifford algebra of $\phi\colon \mathcal{Q}\rightarrow \P^1$. Such Brauer classes of even Clifford algebras arise in the context of derived categories and rationality problems, see, e.g., \cite{ABB14}. We now state and prove another version of Theorem~\ref{maximallinearspace} involving this class \(\alpha_X\):

\begin{thm}\label{thm:clifford-algebra}
Over a field $k$ of characteristic $\neq 2$, fix $g\geq 2$, and let $X$ be a smooth complete intersection of two quadrics in $\P^{2g+1}$.
The following conditions are equivalent:
\begin{enumerate}
\item\label{item:thm-clifford-maximal-plane} $F_{g-1}(X)(k)\neq \emptyset$;
\item\label{item:thm-clifford-br-trivial} $\alpha_X =0$ in $\Br(C)$;
\item\label{item:thm-clifford-rational} $F_{g-2}(X)(k)\neq \emptyset$ and $\mathcal{Q}^{(g-2)}$ is $k$-rational;
\item\label{item:thm-clifford-section} $F_{g-2}(X)(k)\neq \emptyset$ and $\phi^{(g-2)}\colon \mathcal{Q}^{(g-2)}\rightarrow \P^1$ has a section;
\item\label{item:thm-clifford-fano-section} $F_{g-2}(X)(k)\neq \emptyset$ and $p\colon F_1(\mathcal{Q}^{(g-2)}/\P^1)\rightarrow C$ has a section.
\end{enumerate}
\end{thm}
\begin{proof}
 Theorem~\ref{maximallinearspace} shows that \eqref{item:thm-clifford-maximal-plane}\(\Leftrightarrow\)\eqref{item:thm-clifford-rational}, and \eqref{item:thm-clifford-maximal-plane}\(\Leftrightarrow\)\eqref{item:thm-clifford-section} holds by Proposition~\ref{HT}.
 Next, \eqref{item:thm-clifford-section}\(\Leftrightarrow\)\eqref{item:thm-clifford-fano-section} follows from the following geometric argument.
 If $S\subset \mathcal{Q}^{(g-2)}$ is a section for $\phi^{(g-2)}$, then the variety of isotropic lines of $\phi^{(g-2)}$ intersecting $S$ gives a section for $p$.
 Conversely, if $T\subset F_1(\mathcal{Q}^{(g-2)}/\P^1)$ is a section for $p$, define a rational map $T\dashrightarrow \mathcal{Q}^{(g-2)}$ as follows. Recall that \(-1\) denotes the hyperelliptic involution on \(C\). For each (\(\kbar\)-)line \(\lambda\) in a smooth fiber of \(\phi^{(g-2)}\), we send \(\lambda\) to the intersection point \(\lambda \cap (-1)^*\lambda\). This map $T\dashrightarrow \mathcal{Q}^{(g-2)}$ is defined over $k$ because the involution $-1$ is, and the image of \(T\) gives a section for $\phi^{(g-2)}$.

We show \eqref{item:thm-clifford-fano-section}$\Rightarrow$\eqref{item:thm-clifford-br-trivial}.
By \cite[Proposition B.6]{ABB14},
$\alpha_X$ equals the Brauer class corresponding to the smooth conic fibration $p\colon F_1(\mathcal{Q}^{(g-2)}/\P^1)\rightarrow C$.
Since the vanishing of the latter class is equivalent to the existence of the section, \eqref{item:thm-clifford-fano-section} is equivalent to: 
\begin{enumerate}
\setcounter{enumi}{5}
\item\label{item:br-trivial-g-2} $F_{g-2}(X)(k)\neq \emptyset$ and $\alpha_X=0$ in $\Br(C)$.
\end{enumerate}
Clearly, we have \eqref{item:br-trivial-g-2}$\Rightarrow$\eqref{item:thm-clifford-br-trivial},
hence the implication that we want.

It remains for us to show \eqref{item:thm-clifford-br-trivial}$\Rightarrow$\eqref{item:thm-clifford-maximal-plane}.
To achieve this, we use the ``generic point" trick (see \cite[Th\'eor\`eme 5.10]{CT-intersection-quadrics}). 
First, 
if we assume \(F_{g-2}(X)(k)\neq\emptyset\), 
then 
by \eqref{item:thm-clifford-maximal-plane}$\Leftrightarrow$\eqref{item:thm-clifford-fano-section}$\Leftrightarrow$\eqref{item:br-trivial-g-2},
we get $F_{g-1}(X)\neq \emptyset$.
For the general case, we consider the base change \(X\times_k k(F_{g-2}(X))\) to the function field of \(k(F_{g-2}(X))\). Then the previous case implies that \(X\times_k k(F_{g-2}(X))\) contains a \((g-1)\)-plane over \(k(F_{g-2}(X))\); equivalently, the torsor \(F_{g-1}(X)\) splits over \(k(F_{g-2}(X))\). 
But the map \(H^1(k, \bPic^0_C) \to H^1(k(F_{g-2}(X)), \bPic^0_C)\) is injective,
because \(F_{g-2}(X)\) is $\kbar$-rational by Corollary \ref{cor:fano-scheme-rational}\eqref{item:Ci} and any rational map over $k$ from a geometrically rationally connected variety to a torsor under an abelian variety is constant.
Hence \(F_{g-1}(X)\) is split over \(k\), completing the proof.
\end{proof}
\end{rem}

\begin{rem}\label{rem:sym-rational}
By Theorem~\ref{maximallinearspace} and Theorem~\ref{thm:symmetric},
the converse to Theorem~\ref{Fanoscheme} for $(r, N)=(g-2, 2g+1)$ would imply that $k$-rationality of $\mathcal{Q}^{(g-2)}$ is equivalent to $k$-rationality of its \((g-1)\)th symmetric power.
Constructing counterexamples to this latter statement seems to be
a subtle problem.
For instance, for any Severi--Brauer variety $V$ of dimension $n$ over $k$, $\Sym^{n+1}V$ is $k$-rational \cite[Theorem 1.4]{KrashenSaltman}.
However, we can show that $\mathcal{Q}^{(g-2)}$
is not birational to any non-trivial Severi--Brauer variety.
Indeed, by Lemma~\ref{lem:hyperbolic-reduction}\eqref{item:index},
there exists a zero-cycle of degree $1$ on 
$\mathcal{Q}^{(g-2)}$.
Since non-trivial Severi--Brauer varieties never admit a zero-cycle of degree $1$, it is now enough to note that the existence of a zero-cycle of degree $1$ is a birational invariant of smooth projective varieties over a field.
\end{rem}

\section{The even-dimensional case over \(\bb R\)}\label{sec:real-even}

In this section, we focus on the case of even-dimensional \(X \subset \bb P^{2g}\) defined over the real numbers \(\bb R\). First, we prove Theorem~\ref{thm:fano-R-rationality-connected-real-locus}. Then, in Section~\ref{sec:krasnov}, we recall an isotopy invariant that was studied by Krasnov \cite{krasnov-biquadrics} and that was previously used in \cite{HT-intersection-quadrics,HassettKoll'arTschinkel} to study \(\bb R\)-rationality of \(X\). In Section~\ref{sec:isotopy-consequences}, we use this invariant to observe several consequences that the isotopy class has for linear subspaces on \(X\). Finally, we use this to prove Corollary~\ref{cor:classification} and give examples.

\subsection{Rationality criterion for \(F_{g-2}(X)\)}
The maximal linear subspaces that \(X \subset \bb P^{2g}\) contains over \(\bb C\) are \((g-1)\)-planes. We first study the \(\bb R\)-rationality of the Fano scheme \(F_{g-2}(X)\) of second maximal linear subspaces, by proving the following more precise version of Theorem~\ref{thm:fano-R-rationality-connected-real-locus}.

\begin{thm}\label{thm:fano-R-rationality-connected-real-locus-precise}
Over the real numbers, fix $g\geq 2$, and let $X$ be a smooth complete intersection of two quadrics in $\P^{2g}$.
The following are equivalent:
\begin{enumerate}
\item\label{item:real-2nd-maximal-rationality} $F_{g-2}(X)$ is $\R$-rational;
\item\label{item:real-2nd-maximal-connected} $F_{g-2}(X)(\R)$ is non-empty and connected;
\item\label{item:real-2nd-maximal-Q-rationality} $F_{g-2}(X)(\R)$ is non-empty and $\mathcal{Q}^{(g-2)}$ is $\R$-rational;
\item\label{item:real-2nd-maximal-Q-connected} $F_{g-2}(X)(\R)$ is non-empty and $\mathcal{Q}^{(g-2)}(\R)$ is non-empty and connected.
\end{enumerate}
\end{thm}
One key difference between the even- and odd-dimensional cases is that here \(\cal Q^{(g-2)}\) is a surface, whereas it is a threefold when \(\dim X\) is odd (Section~\ref{sec:Q^(r)-construction}). As one might expect, the conclusion of Theorem~\ref{thm:fano-R-rationality-connected-real-locus-precise} fails when \(X\) has odd dimension (see Example~\ref{exmp:odd-Q^g-2-connected-irrational} and the preceding discussion). To prove Theorem~\ref{thm:fano-R-rationality-connected-real-locus-precise}, we will use properties of \(\bb R\)-rationality for surfaces.
We will first need several lemmas about the images of real points under quadric fibrations.

\begin{lem}\label{lem:Q-reallocus-connected}
Let $\phi\colon \mathcal{Q}\rightarrow \P^1$ be a quadric fibration of positive relative dimension over $\R$.
Then $\mathcal{Q}(\R)$ is connected if and only if the image of the induced map $\phi(\R)\colon \mathcal{Q}(\R)\rightarrow \P^1(\R)$ is connected.
\end{lem}
\begin{proof}
The forward direction is immediate.
As for the backward direction, assume that $\mathcal{Q}(\R)$ is disconnected and let $U, V\subset \mathcal{Q}(\R)$ be two disjoint non-empty open subsets that cover and disconnect $\mathcal{Q}(\R)$.
For every $p$ in the image of $\phi(\R)$, the fiber $\phi(\R)^{-1}(p)$ is connected, hence 
is either fully contained in \(U\) or in \(V\). 
This shows $\phi(\R)(U)$ and $\phi(\R)(V)$ are disjoint.
Moreover, $\phi(\R)$ is a closed map since \(\cal Q\) is proper, so
$\phi(\R)(U)$ and $\phi(\R)(V)$ are open in the image of $\phi(\R)$. 
We conclude that the image of $\phi(\R)$ is disconnected by $\phi(\R)(U)$ and $\phi(\R)(V)$.
\end{proof}

\begin{lem}\label{lem:real-locus-symmetric-power}
    Let \(\phi\colon \mathcal{Q} \to \bb P^1\) be a quadric fibration over \(\bb R\). If the image of the induced map \(\phi(\bb R)\colon \mathcal{Q}(\bb R)\to\bb P^1(\bb R)\) on \(\bb R\)-points is disconnected, then so is the image of the induced map \((\Sym^r \mathcal{Q})(\bb R)\to(\Sym^r \bb P^1)(\bb R) \cong \bb P^r(\bb R)\) on real points of the symmetric powers for any \(r\geq 1\). In particular, \((\Sym^r \mathcal{Q})(\bb R)\) is disconnected.
\end{lem}

\begin{proof}
    Let \(\phi^r\colon \Sym^r \cal Q \to \Sym^r \bb P^1 \cong \bb P^r\) be the induced morphism on the symmetric powers, and let \(\varpi\colon \bb P^1\times\dots\times\bb P^1 \to \Sym^r \bb P^1\cong\bb P^r\) be the quotient by the symmetric action.
    Let $I_1,\dots, I_n$ be the connected components of the complement of the image of $\phi(\R)$ in $\P^1(\R)$, and let \(V_i \subset \bb P^r(\bb R)\) denote the image of $I_i\times\P^1(\R)^{r-1}$ under $\varpi(\bb R)\colon \P^1(\R)\times\dots\times\P^1(\R)\rightarrow (\Sym^r\P^1)(\R)\xrightarrow{\sim}\P^r(\R)$.

    First, we claim that the image of \(\phi^r(\bb R)\) is equal to the complement $\P^r(\R)\setminus (\bigcup_{i=1}^n V_i)$. For this, let \(p\in \bb P^r(\bb R)\). If \(p\notin\bigcup_{i=1}^n V_i\), then the fiber of $\phi^r$ over $p$ is a finite product of $\R$-varieties, where each component either is the fiber $\phi^{-1}(q)$ over a real point $q\in \P^1(\R)\setminus (\bigcup_{i=1}^n I_i)$ or is the Weil restriction $R_{\C/\R}(\phi^{-1}(q))$ for some $q\in \P^1(\C)\setminus \P^1(\R)$. In both cases, these varieties have \(\bb R\)-points, so the fiber \((\phi^r)^{-1}(p)\) has an $\R$-point. 
    On the other hand, if $p \in V_i$, then the fiber \((\phi^r)^{-1}(p)\) over $p$ is a finite product of $\R$-varieties, where at least one component is of the form $\phi^{-1}(q)$ for some $q\in I_i$, and hence the fiber has no $\R$-points.
    
    It remains to show that $\P^r(\R)\setminus (\bigcup_{i=1}^n V_i)$ is disconnected. For this, choose points \(s_1\in I_1, s_2\in I_2\). For \(i=1,2\), since \(I_i \subset \bb P^1(\bb R)\) is an open interval, it deformation retracts to \(s_i\). Let \(H_i \subset \bb P^r(\bb R)\) be the image of \(s_i\times \P^1(\R)^{r-1}\) under \(\varpi(\bb R)\). Each \(H_i\) is the real locus of a hyperplane, and \(V_i\) deformation retracts to \(H_i\). Since the complement \(\bb P^r(\bb R)\setminus(H_1 \cup H_2)\) is disconnected, we see that \(\bb P^r(\bb R)\setminus(V_1 \cup V_2)\) is as well. This implies that the image of \(\phi^r(\bb R)\) is disconnected, since the image of \(\phi^r(\bb R)\) is contained in this latter set and intersects each of the two connected components. (For this latter claim, we may reduce to the \(r=2\) case. Then the image of the diagonal \(\Delta\colon \mathbb P^1\to\Sym^2\mathbb P^1\cong\mathbb P^2\) is a smooth conic with real points, and for any \(p\in\bb P^1(\bb R)\), the image of \(p\times \bb P^1(\bb R)\) under the quotient \(\varpi(\bb R)\) is the line \(\ell_p\) tangent to \(\Delta(\mathbb P^1)\) at the point \(\Delta(p)\).
    Then, for each \(i=0,1\), we have \(V_i = \bigcup_{q\in I_i} \ell_q\), and for any \(p\in\bb P^1(\bb R)\setminus(I_1\cup I_2)\), the line \(\ell_p\) properly intersects both \(V_1\) and \(V_2\).)
    \end{proof}

\begin{proof}[Proof of Theorem~\ref{thm:fano-R-rationality-connected-real-locus-precise}]
    We may assume throughout the proof that \(F_{g-2}(X)(\bb R)\neq\emptyset\). First, since \(\cal Q^{(g-2)}\) is a surface, a result of Comessatti \cite[pages 54--55]{Comessatti1913} shows that \eqref{item:real-2nd-maximal-Q-rationality}\(\Leftrightarrow\)\eqref{item:real-2nd-maximal-Q-connected}.
    
    We next show the implications \eqref{item:real-2nd-maximal-rationality}\(\Rightarrow\)\eqref{item:real-2nd-maximal-connected}\(\Rightarrow\)\eqref{item:real-2nd-maximal-Q-rationality}\(\Rightarrow\)\eqref{item:real-2nd-maximal-rationality}.
    First, \eqref{item:real-2nd-maximal-rationality}\(\Rightarrow\)\eqref{item:real-2nd-maximal-connected} is \cite[Theorem 13.3]{DK81}.
    For \eqref{item:real-2nd-maximal-connected}\(\Rightarrow\)\eqref{item:real-2nd-maximal-Q-rationality}, assume \(\cal Q^{(g-2)}\) is irrational over \(\bb R\). Then \cite{Comessatti1913} implies \(\cal Q^{(g-2)}(\bb R)\) is disconnected, so Lemmas~\ref{lem:Q-reallocus-connected} and~\ref{lem:real-locus-symmetric-power} applied to \(\phi^{(g-2)}\colon \cal Q^{(g-2)}\rightarrow \bb P^1\) imply that \((\Sym^{g-1}\cal Q^{(g-2)})(\bb R)\) is disconnected. The Hilbert--Chow morphism \(\rho\colon \Hilb^{g-1}\cal Q^{(g-2)} \to \Sym^{g-1}\cal Q^{(g-2)}\) is a resolution of singularities \cite{Fogarty68}, and the image of \(\rho(\bb R)\) intersects every connected component of \((\Sym^{g-1}\cal Q^{(g-2)})(\bb R)\), so the real locus \((\Hilb^{g-1}\cal Q^{(g-2)})(\bb R)\) is disconnected. The number of real connected components is a birational invariant of smooth projective real varieties \cite[Theorem 13.3]{DK81},
    so by Theorem~\ref{thm:symmetric} this implies \(F_{g-2}(X)(\bb R)\) is disconnected.
    Finally, \eqref{item:real-2nd-maximal-Q-rationality}\(\Rightarrow\)\eqref{item:real-2nd-maximal-rationality} by Theorem~\ref{thm:symmetric} and \cite{Mattuck69}.
\end{proof}

\subsection{Isotopy classification for real complete intersections of quadrics}\label{sec:krasnov}

In this section, we recall an invariant that gives an isotopy classification of real complete intersections of two quadrics in \(\bb P^N\).
This invariant dates back to work of Mordell \cite[\S 3]{Mordell59} and Swinnerton-Dyer \cite[page 268]{SD64} (see also \cite[Section 10]{CTSSD2}).
Using this invariant, Krasnov gave a topological classification of smooth complete intersections of two quadrics when \(N=5, 6\)
\cite{krasnov-biquadrics} (see also \cite[Section 11.2]{HT-intersection-quadrics} and \cite[Section 4.1]{HassettKoll'arTschinkel}).

Over \(\bb R\), let \(X=\{Q_0 = Q_1 = 0\}\subset\bb P^N\) be a smooth complete intersection of two quadrics. The degeneracy locus \(\Delta\) of the pencil \(\cal Q = \{s Q_0 + t Q_1 = 0\} \to \bb P^1\) contains \(r\) real points for some integer \(0 \leq r \leq N+1\). Consider the \(\bb Z_{\geq 0}^2\)-valued function defined by the signatures of the real quadratic forms \[\{s Q_0 + t Q_1 \mid (s,t)\in\bb R^2 \text{ such that }s^2 + t^2 = 1\}\] as \((s,t)\) varies counterclockwise over the unit circle \(\bb S^1 \subset \bb R^2\). This function has \(2r\) points of discontinuity, given by the preimage of \(\Delta(\bb R)\) under the quotient \(\bb S^1 \to \bb P^1(\bb R)\). At each of these points, the number of positive eigenvalues either increases (denoted by \(+\)) or decreases (denoted by \(-\)) by exactly 1.
Since antipodal points have opposite signs, the number of maximal sequences of consecutive \(+\)'s is odd. Thus, we obtain an odd partition
\[r = r_1 + \dots + r_{2u+1}\]
for some non-negative integer $u$,
where each \(r_i\) is the length of a maximal sequence of consecutive \(+\)'s. Following \cite{HassettKoll'arTschinkel}, we call the sequence \((r_1,\ldots,r_{2u+1})\) the \defi{Krasnov invariant} of \(X\). It is well defined up to cyclic permutations and reversal of the order. This invariant determines the rigid isotopy class of \(X\):

\begin{thm}[{\cite[Theorem 1.1]{Krasnov-4-diml}, see also \cite[Appendix A.4.2]{DegtyarevItenbergKharlamov}}]\label{thm:krasnov-isotopy}
    For \(N \geq 3\), isotopy classes of smooth complete intersections of two real quadrics in \(\bb P^N\) correspond to equivalence classes of odd decompositions \(r_1 + \cdots + r_{2u+1} = r\) where \(0\leq r\leq N+1\) is an integer with parity equal to \(N+1\).
\end{thm}
In particular, for each \((r_1,\ldots,r_{2u+1})\) as above, there exist smooth complete intersections of quadrics \(X\subset\bb P^N\) with this given Krasnov invariant.

\begin{defn}\label{defn:krasnov-height-frequency}
    In the above setting, 
    for the Krasnov invariant $(r_1,\cdots, r_{2u+1})$, 
    \cite[Section 2]{krasnov-biquadrics} defines \(I_{\min}\) to be the minimum number of negative eigenvalues that occurs for the quadratic form \(s Q_0 + t Q_1\) as \((s,t)\) varies on the unit circle \(\bb S^1\).
    We define the \defi{height} $h \coloneqq N+1- 2I_{\min}$, and
    we define the \defi{frequency} $f$ to be the number of distinct intervals, i.e., components of \(\bb S^1\) after removing the preimage of \(\Delta(\bb R)\), where \(I_{\min}\) is achieved.
    Note that \(0\leq I_{\min} \leq \lfloor\frac{N+1}{2}\rfloor\), so we always have \(0 \leq h\leq N+1\). If \((s,t)\in\mathbb S^1\) is such that \(s Q_0 + t Q_1\) has \(I_{\min}\) negative eigenvalues, then the signature of \(s Q_0 + t Q_1\) is \((h+I_{\min}, I_{\min})\).
\end{defn}

One can check that the Krasnov invariant uniquely (up to cyclic permutations and order reversal) determines the following sign sequence, where \(+/-\) refers to if the number of positive eigenvalues of \(s Q_0 + t Q_1\) increases/decreases:
\begin{gather*}
    \overbrace{+ \dots +}^{r_i}, \overbrace{- \dots -}^{r_{i+u+1}}, \overbrace{+ \dots +}^{r_{i+1}}, \overbrace{- \dots -}^{r_{i+u+2}}, \overbrace{+ \dots +}^{r_{i+2}}, \dots , \overbrace{- \dots -}^{r_{i+2u}}, \overbrace{+ \dots +}^{r_{i+u}}, \\
    \underbrace{- \dots -}_{r_i}, \underbrace{+ \dots +}_{r_{i+u+1}}, \underbrace{- \dots -}_{r_{i+1}}, \underbrace{+ \dots +}_{r_{i+u+2}}, \underbrace{- \dots -}_{r_{i+2}}, \dots , \underbrace{+ \dots +}_{r_{i+2u}}, \underbrace{- \dots -}_{r_{i+u}}.
\end{gather*}
Here the subscripts are modulo \(2u+1\), so \(h\) and \(f\) in Definition~\ref{defn:krasnov-height-frequency} are well defined.

Notice that \(h > N-1\) if and only if the Krasnov invariant is \((N+1)\).

For an explicit example of the values in Definition~\ref{defn:krasnov-height-frequency}, if \(N=6\), the Krasnov invariant \((2,2,1)\) corresponds (up to cyclic permutations and order reversal) to the sequence of signatures
\[(2,5)\quad (3,4) \quad (4,3) \quad (3,4) \quad (4,3) \quad (5,2) \quad (4,3) \quad (3,4) \quad (4,3) \quad (3,4) .\]
We have \(I_{\min}=2\), \(h=3\), \(f=1\), and the corresponding sign sequence is
\(++,-,++,--,+,--\).

\subsection{Consequences of the isotopy class for linear subspaces}\label{sec:isotopy-consequences}

Hassett--Koll\'ar--Tschinkel 
and Krasnov 
observed that Krasnov invariant determines when \(X\), \(F_1(X)\), and \(F_{g-1}(X)\) have \(\bb R\)-points \cite[Proposition 5.1]{HassettKoll'arTschinkel}, \cite[Theorems 3.1 and 3.6]{krasnov-biquadrics}. We extend their analysis to all linear subspaces on \(X\):

\begin{lem}\label{lem:width-frequency}
Over $\R$, let $X$ be a smooth complete intersection of two quadrics in $\P^N$.
\begin{enumerate}
\item\label{item:fano-width} $F_r(X)(\R)\neq \emptyset$ if and only if $h\leq N-2r-1$.
\item\label{item:reallocus-width-frequency} Assume $h\leq N-2r-1$ and \(N - 2r - 2 \neq 0\). Then $\mathcal{Q}^{(r)}(\R)$ is non-empty and connected if and only if either $h\leq N-2r-3$ or $f=1$.
\end{enumerate}
\end{lem}
\begin{proof}
\eqref{item:fano-width}: 
The proof is by induction on $r$.
For $r=0$, the Amer--Brumer theorem \cite[Theorem 2.2]{Leep} shows that
$X(\R)\neq \emptyset$ if and only if 
$\phi\colon \mathcal{Q}\rightarrow \P^1$ has a section;
this is further equivalent to surjectivity of the induced map \(\phi(\bb R)\) on real points by
a result of Witt \cite[Satz 22]{witt37}.
The latter happens exactly when neither $(N+1,0)$ nor $(0,N+1)$ appears as the signature of a real fiber of $\phi$, which is equivalent to $h\leq N-1$.
For $r>0$, Proposition~\ref{HT} and \cite[Satz 22]{witt37} show that $F_r(X)(\R)\neq\emptyset$ if and only if $F_{r-1}(X)(\R)\neq\emptyset$ and $\phi^{(r-1)}(\bb R)\colon \mathcal{Q}^{(r-1)}(\bb R)\rightarrow \P^1(\bb R)$ is surjective. Surjectivity of \(\phi^{(r-1)}(\bb R)\) happens exactly when $(N-2r+1,0), (0, N-2r+1)$ never appear as the signatures of real fibers of $\phi^{(r-1)}$, which is equivalent to $h\leq N-2r-1$.

\eqref{item:reallocus-width-frequency}: If $h\leq N-2r-1$, then~\eqref{item:fano-width} shows $F_r(X)(\R)\neq\emptyset$.
Hence $\phi^{(r)}\colon\mathcal{Q}^{(r)}\rightarrow \P^1$ is defined and the assumptions imply that \(N - 2r - 2 > 0\), so $\mathcal{Q}^{(r)}(\R)\neq\emptyset$ by Lemma~\ref{lem:hyperbolic-reduction}\eqref{item:index}.
If $h\leq N-2r-3$, then $F_{r+1}(X)(\R)\neq\emptyset$ by \eqref{item:fano-width}, so by Proposition~\ref{HT}, $\mathcal{Q}^{(r)}(\R)$ is $\R$-rational and in particular has non-empty and connected real locus by \cite[Theorem 13.3]{DK81}.  
Finally, if $h=N-2r-1$, then $(N-2r-1,0)$ appears as the signature of a real fiber of $\phi^{(r)}$, and the induced map $\phi^{(r)}(\R)$ on real points is not surjective.
By Lemma~\ref{lem:Q-reallocus-connected}, $\mathcal{Q}^{(r)}(\R)$ is connected if and only if the image of $\phi^{(r)}(\R)$ is connected, and the latter is equivalent to $f=1$.
\end{proof}

Lemma~\ref{lem:width-frequency} and (the proof of) Theorem~\ref{thm:fano-unirational} show that the Krasnov invariant determines \(\bb R\)-unirationality for Fano schemes of non-maximal linear subspaces and the corresponding hyperbolic reductions. That is, for \(X\subset \bb P^N\) and \(0\leq r\leq \lfloor\frac{N}{2}\rfloor -2\), the \(\bb R\)-unirationality of $F_r(X)$ (resp. \(\cal Q^{(r)}\)) is determined by the isotopy class of \(X\).

Furthermore, in the case when \(N=2g\) is even and \(r=g-2\), Theorem~\ref{thm:fano-R-rationality-connected-real-locus-precise} and Lemma~\ref{lem:width-frequency} imply that the isotopy class further determines the following \(\bb R\)-(uni)rationality properties. In particular, we can find the isotopy classes of even-dimensional \(X\) violating the conclusion of Theorem~\ref{maximallinearspace}.
\begin{cor}\label{cor:fano-R-rationality-connected-real-locus-isotopy}
    Over \(\bb R\), fix \(g \geq 2\), and let \(X\) be a smooth complete intersection of two quadrics in \(\bb P^{2g}\).
    \begin{enumerate}
        \item\label{item:Q-g-2-even-isotopy} Then \(\bb R\)-rationality of \(F_{g-2}(X)\) (resp. $\mathcal{Q}^{(g-2)}$) is determined by the isotopy class of \(X\).
        \item\label{item:even-maximallinearspace-fails} \(F_{g-1}(X)(\bb R)=\emptyset\), \(F_{g-2}(X)(\R)\neq\emptyset\), and \(\cal Q^{(g-2)}\) is \(\bb R\)-rational if and only if \(h=3\) and \(f=1\).
        \item\label{item:even-F_g-2-uni-but-irrational} \(F_{g-2}(X)\) is \(\bb R\)-unirational but not \(\bb R\)-rational if and only if \(h=3\) and \(f>1\).
    \end{enumerate}
\end{cor}
In particular, for any \(g \geq 2\), there exist smooth complete intersections of quadrics \(X\subset\bb P^{2g}\) with Krasnov invariant \((3)\) by Theorem~\ref{thm:krasnov-isotopy}, so examples satisfying Corollary~\ref{cor:fano-R-rationality-connected-real-locus-isotopy}\eqref{item:even-maximallinearspace-fails} exist in any even dimension.

\begin{rem}
    In general, for \(N \geq 7\), it is not known whether the isotopy class of \(X\) determines \(\bb R\)-rationality of \(F_r(X)\). See \cite[Section 6.2]{HassettKoll'arTschinkel} for an isotopy class in the case \(N=8\) where the \(\bb R\)-rationality of its members is unknown.
\end{rem}

As a sample application of Lemma~\ref{lem:width-frequency}, for the case \(N=6\) we show the following properties for the Fano schemes \(F_r(X)\), extending the analysis in \cite{HassettKoll'arTschinkel}. One could similarly carry out an analysis for any \(N\).

\begin{cor}\label{cor:classification-4-fold}
Over $\R$, let $X$ be a smooth complete intersection of two quadrics in $\P^6$.
\begin{enumerate}
\item\label{item:classification-cor-F_2} $F_2(X)(\R)$ is non-empty if and only if the Krasnov invariant is one of $(1)$, $(1,1,1)$, $(1,1,1,1,1)$, or $(1,1,1,1,1,1,1)$.
\item\label{item:classification-cor-F_1-Q} $F_1(X)(\R)$ is non-empty and 
$\mathcal{Q}^{(1)}(\R)$ is non-empty and connected
if and only if the Krasnov invariant is one of those listed in item~\ref{item:classification-cor-F_2}, or is $(3)$, $(2,2,1)$, or $(2,1,2,1,1)$.
\item\label{item:classification-cor-F_1}
$F_1(X)(\R)$ is non-empty 
if and only if
Krasnov invariant is one of those listed in item~\ref{item:classification-cor-F_1-Q}, or is
$(3,1,1)$, $(3,2,2)$, $(3,1,1,1,1)$, or $(2,2,1,1,1)$.
\item\label{item:classification-cor-X} $X(\R)$ is non-empty and 
$\mathcal{Q}^{(0)}(\R)$ is non-empty and connected
if and only if the Krasnov invariant is one of those listed in item~\ref{item:classification-cor-F_1}, or is $(5)$, $(4,2,1)$, or $(3,3,1)$.
\item\label{item:classification-r-point} 
$X(\R)$ is non-empty if and only if 
the Krasnov invariant is one of those listed in item~\ref{item:classification-cor-X} or is $(5,1,1)$.
\end{enumerate}
\end{cor}
\begin{proof}
There are only finitely many possible Krasnov invariants for a given smooth complete intersection of two real quadrics.
For the list when $N \leq 8$ is even, see the \(r \leq 7\) cases in \cite[Section 4.1, Figure 1]{HassettKoll'arTschinkel}.
Lemma~\ref{lem:width-frequency} and a direct computation then conclude the proof.
\end{proof}

\begin{proof}[Proof of Corollary~\ref{cor:classification}]
The result follows by combining 
Corollary~\ref{cor:classification-4-fold} with \cite[Remark 3.28.3]{CTSSD},
\cite[Theorem 1.1]{HassettKoll'arTschinkel}, and
Theorems~\ref{thm:fano-unirational} and~\ref{thm:fano-R-rationality-connected-real-locus-precise}.
Since $\mathcal{Q}^{(0)}$ is $\R$-birational to $X$ (Lemma~\ref{lem:hyperbolic-reduction}\eqref{item:hyperbolic-reduction-max-min}), $\mathcal{Q}^{(0)}(\R)$ is non-empty and connected if and only if $X(\R)$ is by \cite[Theorem 13.3]{DK81}.
\end{proof}

Finally, we end the paper by returning to odd-dimensional \(X\subset\bb P^{2g+1}\). For any \(g \geq 2\), Theorem~\ref{maximallinearspace} and Lemma~\ref{lem:width-frequency}\eqref{item:fano-width} and \eqref{item:reallocus-width-frequency} imply that Krasnov invariants with \(h=4\) and \(f=1\) correspond to \(X\subset\bb P^{2g+1}\) where \(\cal Q^{(g-2)}\) has non-empty and connected real locus but is irrational over \(\bb R\). In particular, applying Theorem~\ref{thm:krasnov-isotopy} to the isotopy class \((4)\), we see that Theorem~\ref{thm:fano-R-rationality-connected-real-locus} fails in every odd dimension. For concreteness, we list the possible Krasnov invariants here for \(g=2,3\):

\begin{exmp}\label{exmp:odd-Q^g-2-connected-irrational}
    For \(g=2\) and \(N=5\), the Krasnov invariants with \(h=4\) and \(f=1\) are \((4)\), \((4,1,1)\), and \((3,2,1)\). For \(g=3\) and \(N=7\), the Krasnov invariants with \(h=4\) and \(f=1\) are \((4)\), \((4,1,1)\), \((3,2,1)\), \((3,1,2,1,1)\), \((2,2,2,1,1)\), and \((3,3,2)\). For each of these, \(\cal Q^{(g-2)}(\bb R)\neq\emptyset\) is connected but \(F_{g-1}(X)(\bb R)=\emptyset\), so \(\cal Q^{(g-2)}\) is irrational over \(\bb R\).
\end{exmp}

\bibliographystyle{alpha}
\bibliography{references.bib}

\end{document}